\documentclass[reqno,10pt]{article}
\usepackage[T1]{fontenc} 
\usepackage[utf8]{inputenc}
\usepackage{microtype} 
\usepackage{lipsum}
\usepackage{titling}		
\usepackage[italian,english]{babel} \usepackage{indentfirst}
\usepackage{url} 
\usepackage{pdfpages}

\usepackage{bm} 
\usepackage{amsthm}
\usepackage{bbm} 
\usepackage{dsfont}
\usepackage{relsize}
\usepackage{mathtools}
\usepackage{mathrsfs}  
\usepackage{array}
\usepackage{amsmath,amsfonts,latexsym,amsthm,stackrel}
\usepackage[mathcal]{euscript}

\usepackage[pdfpagelabels]{hyperref}
\hypersetup{
		colorlinks=true,
		linkcolor=blue,
		anchorcolor=black,
		citecolor=green,
		urlcolor=red}

\counterwithin{equation}{section}
\usepackage[noabbrev]{cleveref}
\usepackage[psamsfonts]{amssymb}
\usepackage{faktor} 
\usepackage[new]{old-arrows}
\usepackage{enumitem}

\usepackage[autostyle]{csquotes}

\usepackage[
    backend=biber,
    style=numeric-comp,
    sortlocale=de_DE,
    natbib=true,
    url=false, 
    doi=true,
    eprint=false
]{biblatex}
\makeatletter

\newtheorem{thm}{Theorem}[section]
\newtheorem{prop}[thm]{Proposition}

\newtheorem{lem}[thm]{Lemma}
\newtheorem{cor}[thm]{Corollary}

\theoremstyle{definition}
\newtheorem{dfn}[thm]{Definition}

\theoremstyle{remark}
\newtheorem{oss}[thm]{Remark}

\theoremstyle{definition}

\newtheorem{ip}[thm]{Hypotesis}

\newcommand{\iprod}[2]{\left\langle#1,#2\right\rangle}
\newcommand{\re}[1]{\mbox{\mdseries{Range}}(#1)}
\newcommand{\norm}[1]{\left\lVert#1\right\rVert}
\newcommand{\abs}[1]{\left\lvert#1\right\rvert}

\newcommand{\Nat}{\mathbb{N}}

\newcommand{\R}{\mathbb{R}}

\newcommand{\al}{\alpha}
\newcommand{\be}{\beta}
\newcommand{\g}{\gamma}
\newcommand{\ep}{\varepsilon}

\newcommand{\si}{\sigma}

\newcommand{\ph}{\varphi}
\newcommand{\la}{\lambda}
\newcommand{\te}{\theta}
\newcommand{\La}{\Lambda}

\newcommand{\ca}[1]{\mathbbm{1}_{#1}}
\newcommand{\de}{\partial}
\newcommand{\sm}{\smallsetminus}
\newcommand{\op}{\mathcal{L}}
\newcommand{\Ker}{{\rm Ker}}

\newcommand{\Tr}[1]{{\rm Tr}\left(#1\right)}

\newcommand{\cm}{\mathcal{H}}
\newcommand{\Ng}{\mathcal{N}}
\newcommand{\ac}{\mathrm{H}}

\newcommand{\mts}{m^x(t,s)}
\newcommand{\uts}{U(t,s)}
\newcommand{\utz}{U(t,\sigma)}
\newcommand{\ets}{E}
\newcommand{\pst}{P_{s,t}}
\newcommand{\Lats}{\Lambda(t,s)}
\newcommand{\Qm}[1]{Q(#1)^{-\frac{1}{2}}}
\newcommand{\Qp}[1]{Q(#1)^{\frac{1}{2}}}
\newcommand{\Qum}{Q^{-\frac{1}{2}}}
\newcommand{\Qup}{Q^{\frac{1}{2}}}

\newcommand{\liph}[1]{{\rm Lip_E} (#1)}
\newcommand{\ix}[2]{\left\langle#1,#2\right\rangle_X}

\newcommand{\os}{\mathcal{O}}
\newcommand{\oa}{\mathcal{A}}

\newcommand{\id}{{\rm Id}}

\usepackage{scalerel,stackengine}
\stackMath
\newcommand\rwhat[1]{%
\savestack{\tmpbox}{\stretchto{%
  \scaleto{%
    \scalerel*[\widthof{\ensuremath{#1}}]{\kern-.6pt\bigwedge\kern-.6pt}%
    {\rule[-\textheight/2]{1ex}{\textheight}}
  }{\textheight}%
}{0.5ex}}%
\stackon[1pt]{#1}{\tmpbox}%
}
\title{ON SMOOTHING IN NON AUTONOMOUS ORNSTEIN-UHLENBECK EQUATIONS IN INFINITE DIMENSIONS}
\author{\textit{Paolo De Fazio}\\University of Parma}
\date{}

\begin{document}

\maketitle

\begin{abstract}
We prove smoothing properties along suitable directions of the Ornstein-Uhlenbeck evolution operator, namely the evolution operator associated to non autonomous Ornstein-Uhlenbeck equations. Moreover we use such smoothing estimates to prove Schauder type theorems, again along suitable directions, for the mild solutions of a class of evolution equations.
\end{abstract}


\section[]{Introduction}

Let $(X, \ix{\cdot}{\cdot},\norm{\,\cdot\,}_X)$ be a separable Hilbert space, $T>0$ and $\Delta=\{(s,t)\in[0,T]^2\ |\ s<t\}$. We consider the mild solution 
\begin{equation}\label{solmild}
u(s,x):=\pst\ph(x)-\int_s^t\bigl(P_{s,\si}\psi(\si,\cdot)\bigr)(x)\,d\si,\ \ \ph\in C_b(X),\ \psi\in C_b\bigl([0,t]\times X\bigr)
\end{equation}
to a class of backward non autonomous initial value problems,
\begin{equation}\label{bwp}
\begin{cases}
\de_s u(s,x)+ L(s) u(s,x)=\psi(s,x),\ \ \ \ (s,t)\in\Delta,\  x\in X,\\
u(t,x)=\ph(x),\ \ x\in X,
\end{cases}
\end{equation}
where the operators $L(t)$ are of Ornstein-Uhlenbeck type
\begin{equation}
L(t)\ph(x)=\frac{1}{2}\mbox{\rm Tr}\Bigl(Q(t)D^2\ph(x)\Bigr)+\ix{A(t)x+f(t)}{\nabla\ph(x)},
\end{equation} 
the family $\{A(t)\}_{t\in[0,T]}$ generates a strongly continuous evolution operator and $Q(t)$ is a self-adjoint nonnegative operator for every $t\in[0,T]$.

For any $t\in[0,T]$ the evolution family $\{P_{s,t}\}_{s\in[0,t]}$ is defined by
\begin{align}
 P_{t,t}&=I\ \ \mbox{for any}\ \ t\in[0,T], \\
P_{s,t}\ph(x)&=\int_X\ph(y)\,\Ng_{m^x(t,s),Q(t,s)}(dy)\nonumber \\&=\int_X\ph(y+m^x(t,s))\,\Ng_{0,Q(t,s)}(dy), \ \ (s,t)\in\overline{\Delta},\ \ph\in C_b(X)
\end{align}
where we denote by $\Ng_{m,Q}$ the Gaussian measure in $X$ with mean $m$ and covariance $Q$. Here, for any $(s,t)\in\overline{\Delta}$ we have
\begin{align} \label{qtscov}
Q(t,s)&=\int_s^tU(t,r) Q(r) U(t,r)^\star\,dr,\\
m^x(t,s)&=U(t,s)x+g(t,s),\\
g(t,s)&=\int_s^tU(t,r) f(r)\,dr.
\end{align}
where $\{U(t,r)\}_{(t,r)\in\overline{\Delta}}$ is the strongly continuous evolution operator in $X$ associated to the family $\{A(t)\}_{t\in[0,T]}$. Under minimal assumptions, the mapping $$(s,\si,x)\in\overline{\Delta}\times X\longmapsto \bigl(P_{s,\si}\psi(\si,\cdot)\bigr)(x)\in\R$$ is measurable (e.g. \cite[Lemma\,2.3]{cerlun}). So the integral in \eqref{solmild} is well defined. In this infinite dimensional setting, we need that the operators $Q(t,s)$ defined by \eqref{qtscov} have finite trace.

Let $Y$ be a separable Hilbert space and $\{B(t)\}_{t\in[0,T]}\subseteq\op(Y;X)$. We consider the following forward  stochastic differential equation
\begin{equation}\label{stocpb}
\begin{cases}
dX_t(s,x)=\bigl(A(t)X_t(s,x)+f(t)\bigr)dt+B(t)dW_t,\ \ (s,t)\in\Delta,\\
X_s(s,x)=x\in X,
\end{cases}
\end{equation}
where $W_t$ is a $Y$-valued cylindrical Wiener process. If $\psi\equiv 0$ and $Q(t)=B(t)B(t)^\star$, \eqref{bwp} is the Kolmogorov equation formally associated to \eqref{stocpb}. Namely, it is the equation formally satisfied by $$u(s,x)=\mathbb{E}\bigl(\ph(X_t(s,x))\bigr),\ \ (t,s)\in\overline{\Delta}.$$
For a proof of this fact in the autonomous case see \cite{MR3236753}.

We are interested in smoothing properties of $\pst$ along some normed spaces $E$, continuously embedded in $X$. Fixed $(s,t)\in\Delta$, such properties will be proved  under the assumpion
\begin{equation}\label{16}
 U(t,s)(E)\subseteq\cm_{t,s}:=\Qp{t,s}(X)\ \mbox{and}\ U(t,s)_{|_E}\in\op(E,\cm_{t,s}). 
 \end{equation}
 In this case we define the operator $$\La(t,s)=\Qm{t,s} U(t,s),\ \ (s,t)\in\Delta,$$ where $\Qm{t,s}$ is the pseudo-inverse of $\Qp{t,s}$.

If \eqref{16} holds, we prove that for all $k\in\Nat$, $\pst$ maps $C_b(X)$ (the space of continuous and bounded functions from $X$ to $\R$, see Sect. 2) into $C^k_E(X)$ (the space of the functions from $X$ to $\R$ $k$-times Frèchet differentiable along $E$ having bounded Frèchet differentials along $E$ see Sect. 2). Moreover, we give an explicit formula for the Frèchet derivatives of $\pst$ of any order along $E$, that involves the operators $\La(t,s)$ and allows to estimate such derivatives. 

If $E$ is such that \eqref{16} holds for every $(s,t)\in\Delta$, in general $\norm{\La(t,s)}_{\op(E,X)}$ blows up as $t-s\rightarrow 0^+$ and we assume in addition that $\norm{\La(t,s)}_{\op(E,X)}$ has a powerlike bound, namely there exist $\theta, C>0$ such that
\begin{equation}\label{1.10}
\norm{\La(t,s)}_{\op(E,X)}\leq \frac{C}{(t-s)^\theta},\ \ (s,t)\in\Delta.
\end{equation}
In this case we prove H\"older maximal regularity of \eqref{solmild} along $E$. More precisely, if  $\al\in[0,1)$, $\displaystyle{\al+\frac{1}{\te}\notin\Nat}$, $\ph\in C^{\al+\frac{1}{\te}}_E(X)$ and $\psi\in C^{0,\al}_E\bigl([0,t]\times X\bigr)$, then $u(s,\cdot)\in C_E^{0,\al+\frac{1}{\te}}\bigl([0,t]\times X\bigr)$. Moreover there exists $C=C(T,\al)>0$, independent of $\ph$ and $\psi$, such that 
\begin{equation}
\norm{u}_{C_E^{0,\al+\frac{1}{\te}}([0,t]\times X)}\leq C\bigl(\norm{\ph}_{C^{\al+\frac{1}{\te}}_E(X)}+\norm{\psi}_{C_E^{0,\al}([0,t]\times X)}\bigr).
\end{equation}
If $\displaystyle{\al+\frac{1}{\te}}$ is an integer, $u(s,\cdot)$ only belongs to the Zygmund space $Z_E^{\al+\frac{1}{\te}}(X)$  for every $s$. For the definitions of the spaces $C_E^{\gamma}(X)$, $Z_E^{k}(X)$ and $C_E^{0,\gamma}([0,t]\times X)$ see Sect. 2. Zygmund regularity in place of $C^k$ regularity is not due to the infinite dimensional setting nor to the time dependence of the data. Indeed, we have the same result even in finite dimension for the heat equation $u_s+\Delta u=\psi$.

In the autonomous case, the first Schauder estimates for Ornstein-Uhlenbeck type equations were proven in \cite{MR1343161, MR1475774} in finite dimension and in \cite{MR1430142} in some infinite dimensional equations.  Maximal H\"older regularity along suitable directions was proved in \cite{MR4311102} for a general class of autonomous equations.

Non autonomous Ornstein Uhlenbeck equations in infinite dimension were studied in \cite{cerlun} under the assumption that $\uts(X)\subseteq \Qp{t,s}(X)$. In this case, $\pst$ is strong Feller, namely it maps maps $B_b(X)$ (the set of bounded Borel measurable functions) into $C_b(X)$.  In fact, it maps $B_b(X)$ into $C^k_b(X)$ for every $k\in\Nat$ and, assuming that \eqref{1.10} holds for $E=X$, maximal H\"older regularity for \eqref{solmild} holds along any direction. A more general class of non autonomous evolution families $\pst$ associated to stochastic differential equations with Levy noise are studied in \cite{MR3466574} again in the Strong Feller case.

It is well known that in infinite dimension the classical Ornstein-Uhlenbeck semigroup is smoothing only along suitable directions, see \cite{MR1642391}, as well as the heat semigroup generated by the Gross Laplacian, see \cite{MR1985790}.

Others authors looked for smoothing results along suitable directions for perturbations of autonomous Ornstein-Uhlenbeck semigroups. For further readings we refer to \cite{Bignamini2022, Bignamini2021, bigna22} and \cite{MR2299922}. 

In section 2 we study the smoothing properties of $\pst$. In particular we prove that, if \eqref{16} holds, $\pst$ maps $C_b(X)$ in $C^k_E(X)$ for every $k\in\Nat$ and moreover there exists a constant $C_k>0$ such that 
\begin{equation}\label{stimanormediff}
\sup_{x\in X}\norm{D_E^k(\pst\ph)(x)}_{\op^k(E)}\leq C_k\norm{\Lats}^k_{\op(E; X)}\norm{\ph}_\infty.
\end{equation}
In addition we extend to our non autonomous case a result of  \cite[Sect.\,3]{MR1976297}, giving sufficient conditions in order that $E=\Qp{s}(X)$ satisfies \eqref{16} for $s\in(0,t)$.

In section 3 we prove the above maximal H\"older regularity results.

Section 4 concerns four  genuinely non autonomous examples. In the first example $A(t)$ and $B(t)$ are diagonal operators with respect to the same Hilbert basis of $X$. This allows to give easily necessary and sufficient conditions for \eqref{16} and \eqref{1.10} to hold.

In the second example we consider $A(t)=a(t) I$, where $a$ is a continuous function. We get a non autonomous version of the Ornstein-Uhlenbeck semigroup used in the Malliavin calculus and we extend to such non autonomous case the results of \cite{MR4011050}.

In the third example the operator $A(t)$ is the realization of a second order elliptic differential operator in $X=L^2(\os)$ with Dirichlet, Neumann or Robin boundary conditions and smooth enough coefficients, $\os$ is a bounded open smooth subset of $\R^d$. 

In \cite{cerlun}, the authors proved Schauder type results for the Dirichlet problem, working in H\"older spaces with increments along the whole $X$ in which case $\theta\geq\frac{1}{2}$. In the present paper we show that we can gain more regularity along suitable subspaces of $X$. A special case is given by $d=1$, $A(t)$ is a second order elliptic differential operator with Dirichlet boundary conditions and $B(t)=\mbox{Id}$. Indeed choosing $E=(L^2(I),H^2(I)\cap H^{1}_0(I))_{\al,p} $ with $\al\in\bigl(0,\frac{1}{2}\bigr)$ and $p\in[1,+\infty)$, \eqref{16} holds and \eqref{1.10} is satisfied with $\theta=\frac{1}{2}-\al$. So $E$ is a suitable Besov space with possible boundary condition.
Since in the above mentioned Schauder theorems we gain $\frac{1}{\theta}$ degrees of regularity of \eqref{solmild} along $E$, we find different kinds of smoothing according to the chosen space. This phenomenon is similar to what happens in finite dimension for hypoelliptic Ornstein-Uhlenbeck operators in the autonomous case (see \cite{MR1475774}). 

Smoothing estimates such as \eqref{stimanormediff} are basic tools for several development of the theory besides H\"older maximal regularity. For instance we plan to use them in the study of hypercontractivity properties of $\pst$ in $L^p$ spaces with respect to suitable measures. 

\section[]{Notations and assumptions}

If $X$ and $Y$ are real Banach spaces we denote by $\op(X;Y)$ the space of bounded linear operators from $X$ to $Y$. If $X=Y$, we write $\op(X)$ instead of $\op(X;X)$ and if $Y=\R$ we simply write $X^\star$ instead of $\op(X;\R)$. For $k\geq 2$, $\op^{k}(X)$ is the space of the $k$-linear bounded operators $T: X^k\longrightarrow\R$ endowed with the norm $$\norm{T}_{\op^{k}(X)}=\sup\biggl\{\frac{\abs{T(x_1,...,x_k)}}{\norm{x_1}_X\cdot\cdot\cdot\norm{x_k}_X}:\ x_1,...,x_k\in X\sm\{0\}\biggr\}.$$By $B_b(X;Y)$ and $C_b(X;Y)$ we denote the space of bounded Borel functions from $X$ to $Y$ and the space of bounded and continuous functions from $X$ to $Y$, respectively. We endow them with the sup norm $$\norm{F}_\infty=\sup_{x\in X}\norm{F(x)}_Y.$$  If $Y=\R$, we simply write $B_b(X)$ and $C_b(X)$ instead of $B_b(X;\R)$ and $C_b(X;\R)$, respectively.
Let $(E,\norm{\, \cdot\, }_{E})$ be a normed space such that $E\subseteq X$ with continuous embedding. For $\al\in(0,1)$ we define the H\"older spaces along $E$ as $$C^\al_E(X;Y)=\biggl\{F\in C_b(X;Y):\ [F]_{C^\al_E(X;Y)}=\sup_{{\substack{x\in X,\\ h\in E\sm\{0\}}}}\frac{\norm{F(x+h)-F(x)}_Y}{\norm{h}_E^\al}<+\infty\biggr\},$$
$$\norm{F}_{C^\al_E(X;Y)}=\sup_{x\in X}\norm{F(x)}_{Y}+[F]_{C^\al_E(X:Y)},$$ 
and the Lipschitz space along $E$ as $$\liph{X;Y}=\biggl\{F\in C_b(X;Y):\ [F]_{\liph{X;Y}}=\sup_{\substack{x\in X,\\ h\in E\sm\{0\}}}\frac{\norm{F(x+h)-F(x)}_Y}{\norm{h}_E}<+\infty\biggr\},$$ $$\norm{F}_{\liph{X;Y}}=\sup_{x\in X} \norm{F(x)}_Y+[F]_{\liph{X;Y}}.$$
Again, if $Y=\R$ we write $C^\al_E(X)$ and $\liph{X}$ instead of $C^\al_E(X;\R)$ and $\liph{X;\R}$, respectively.
Moreover we say that a map $f:X\longrightarrow\R$ is $E$-G\^{a}teaux differentiable at $x\in X$ if there exists a bounded linear operator $l_x:E\longrightarrow\R$ such that for any $h\in E$ we have 
\begin{equation}
\lim_{t\rightarrow 0}\frac{f(x+th)-f(x)-l_x(h)}{t}=0.
\end{equation}
$l_x$ is the $E$-G\^{a}teaux differential of $f$ at $x$ and we set $l_x=D^G_Ef(x)$. 
We say that a map $f:X\longrightarrow\R$ is $E$-Fréchet differentiable at $x\in X$ if there exists a bounded linear operator $t_x:E\longrightarrow\R$ such that 
\begin{equation}
\lim_{\norm{h}_{E}\rightarrow 0}\frac{f(x+h)-f(x)-t_x(h)}{\norm{h}_{E}}=0.
\end{equation}
$t_x$ is the $E$-Fréchet differential of $f$ at $x$ and we set $t_x=D_Ef(x)$. Clearly, if $f$ is Fréchet differentiable at $x$ then $f$ is $E$-Fréchet differentiable at $x$ and this is due to the continuous embedding of $E$ in $X$. 
More generally, for any $F:X\longrightarrow Y$, we say that $F$ is $E$-G\^{a}teaux differentiable at $x$ if there exists $L_x\in\op(E,Y)$ such that for any $h\in E$, we have
\begin{equation}
Y-\lim_{t\rightarrow 0}\frac{F(x+th)-F(x)-L_x(h)}{t}=0.
\end{equation}
$L_x$ is the $E$-G\^{a}teaux differential of $F$ at $x$ and we denote it by $D^G_E F(x)$. We say that $F$ is $E$-Fréchet differentiable at $x$ if there exists $T_x\in\op(E,Y)$ such that
\begin{equation}\label{yfrechet}
Y-\lim_{\norm{h}_{E}\rightarrow 0}\frac{F(x+h)-F(x)-T_x(h)}{\norm{h}_{E}}=0.
\end{equation}
$T_x$ is the $E$- Fréchet differential of $F$ at $x$ and we denote it by $D_E F(x)$. Clearly, if $E=X$ the notions of $E$-Fréchet differentiability and $E$-G\^{a}teaux differentiability coincide with the usual ones. In this case we omit the subindex $E$ in the notations above. For instance, we write $D F$ instead of $D_E F$ for the Fréchet derivative.

If $f:X\longrightarrow \R$ is $E$-Fréchet differentiable at $x$, we say that $f$ is twice $E$-Fréchet differentiable at $x$ if $D_E f:X\longrightarrow E^\star$ is $E$-Fréchet differentiable at $x$. 
 Hence, the second order differential $D^2_E f\in\op^{2}(E)$ of $f$ is defined  by $$D^2_E f(x)(k,h):=(T_x k)(h),$$ where $T_x$ is the operator in \eqref{yfrechet} with $F(x)$ replaced by $D_E f$ and $Y=E^\star$.
 
If $f$ is $(k-1)$-times $E$-Fréchet differentiable at $x$ with $k\geq 2$, we say that $f$ is $k$-times $E$-Fréchet differentiable at $x$ if $D^{k-1}_E f: X\longrightarrow \op^{k-1}(E)$ is $E$-Fréchet differentiable at $x$.  In this case $D^k_E f\in\op^k(X)$ is defined as $$D_E^k f(x)(h_1,...,h_k):=(T_xh_1)(h_2,...,h_k)$$ where $T_x$ is the operator in \eqref{yfrechet} with $F(x)$ replaced by $D_E^{k-1}f(x)$ and $Y=\op^{k-1}(E)$. 

For any $k\in \Nat$, $C^k_{E}(X)$ is the subspace of $C_b(X)$ consisting of all functions $f:X\longrightarrow \R$ $k$-times $E$-Fréchet differentiable at any point with $D_E^j f$ continuous and bounded from $X$ to $\op^{j}(E)$ for $j\leq k$. $C^k_E(X)$ is endowed with the norm 
\begin{equation}
\norm{f}_{C^k_E(X)}:= \norm{F}_\infty+\sum_{j=1}^k\sup_{x\in X}\norm{D_E^j f(x)}_{\op^{j}(E)}.
\end{equation}

$Z^1_E(X;Y)$ is the Zygmund space along $E$. It  is defined by 
\begin{align}
Z^1_E(X;Y)=&\biggl\{F\in C_b(X;Y)\ :\ \nonumber\\
&[F]_{Z^1_E(X;Y)}=\sup_{\substack{x\in X,\\ h\in E\sm\{0\}}}\frac{\norm{F(x+2h)-2F(x+h)+F(x)}_Y}{\norm{h}_E}<+\infty\biggr\},\nonumber
\end{align} and it is endowed with the norm $$\norm{F}_{Z^1_E(X;Y)}=\norm{F}_\infty+[F]_{Z^1_E(X;Y)}.$$
Higher order H\"older and Zygmund spaces along $E$ are defined as follows. For $\al\in(0,1)$ and $n\in\Nat$, we set $$C^{n+\al}_E(X):=\biggl\{f\in C^n_E(X)\ :\ D^n_E f\in C^\al_E(X;\op^n(E))\biggr\},$$ $$\norm{f}_{C^{n+\al}_E(X)}:=\norm{f}_{C^{n}_E(X)}+[D^n_E f]_{C^\al_E(X;\op^n(E))}$$ and for $n\geq 2$, $$Z^n_E(X)=\biggl\{f\in C^{n-1}_E(X)\ :\ D^{n-1}_E f\in Z^1_E(X;\op^n(E))\biggr\},$$ $$\norm{f}_{Z^n_E(X)}:=\norm{f}_{C^{n-1}_E(X)}+[D^{n-1}_E F]_{Z^1_E(X;\op^{n-1}(E))}.$$
Clearly if $E=X$ the functional spaces above coincide with the usual ones. 

Finally we introduce spaces of functions depending both on time and space variables. For every $a,b\in\R$, $a<b$ and $\al>0$, we denote by $C^{0,\al}_E\bigl([a,b]\times X\bigr)$ the space of all bounded continuous functions $\psi:[a,b]\times X\longrightarrow \R$ such that $\psi(s,\cdot)\in C^\al_E(X)$, for every $s\in [a,b]$, with $$\norm{\psi}_{C^\al_E([a,b]\times X)}=\sup_{s\in[a,b]}\norm{\psi(s,\cdot)}_{C^\al_E(X)}<+\infty.$$
If $\al\geq 1$ we also require that the mappings $$(s,x)\longrightarrow \frac{\de\psi}{\de h_1...\de h_k}(s,x)$$ are continuous in $[a,b]\times X$, for every $h_1,...,h_k\in E$ with $k\leq [\al]$.
Now, for every $k\in\Nat$ we denote by $Z^{0,k}_E\bigl([a,b]\times X\bigr)$ the space of all bounded continuous functions $\psi:[a,b]\times X\longrightarrow \R$ such that $\psi(s,\cdot)\in Z^k_E(X)$, for every $s\in [a,b]$, with $$\norm{\psi}_{Z^{0,k}_E([a,b]\times X)}=\sup_{s\in[a,b]}\norm{\psi(s,\cdot)}_{Z^k_E(X)}<+\infty.$$
If $k\geq 2$, we also require that the mapping $$(s,x)\longrightarrow \frac{\de\psi}{\de h_1...\de h_i}(s,x)$$ are continuous in $[a,b]\times X$, for every $h_1,...,h_i\in E$ with $i\leq k-1$.

Let $(X,\norm{\cdot}_X,\ix{\cdot}{\cdot})$ is a separable Hilbert space. We say that $Q\in\mathcal{L}(X)$ is \emph{non-negative} (\emph{positive}) if for every $x\in X\sm\{0\}$
\[
\langle Qx,x\rangle_X\geq 0\ (>0).
\]
On the other hand, $Q \in \mathcal{L}(X)$ is a \emph{non-positive} (respectively \emph{negative}) operator if $-Q$ is non-negative (respectively positive). Let $Q\in\mathcal{L}(X)$ be a non-negative and self-adjoint operator. We say that $Q$ is a trace class operator if
\begin{align}\label{trace_defn}
\Tr{Q}:=\sum_{n=1}^{+\infty}\langle Qe_n,e_n\rangle_X<+\infty,
\end{align}
for some (and hence for all) orthonormal basis $\{e_n\}_{n\in\Nat}$ of $X$. We recall that the trace operator, defined in \eqref{trace_defn}, is independent of the choice of the orthonormal basis.

Let $\Delta=\{(s,t)\in[0,T]^2\ |\ s<t\}$ and let $X, Y$ be two separable Hilbert spaces. Our basic assumptions are the following.
\begin{ip}\label{1}
\leavevmode
\begin{enumerate} 
\item[(1)] $\{U(t,s)\}_{(s,t)\in\overline{\Delta}}\subseteq\op(X)$ is a strongly continuous evolution operator, namely for every $x\in X$ the map
\begin{equation}
(s,t)\in\overline{\Delta}\longmapsto U(t,s)x\in X,
\end{equation}
is continuous and
\begin{enumerate}
\item $U(t,t)=I$ for any $t\in[0,T]$,
\item $U(t,r)U(r,s)=U(t,s)$ for $0\leq s\leq r\leq t\leq T$.
\end{enumerate}
\item[(2)] The family of operators $\{B(t)\}_{t\in[0,T]}\subseteq\op(Y;X)$ is bounded and strongly measurable, namely
\begin{enumerate}
\item there exists $K>0$ such that 
\begin{equation}\label{bshyp}
\sup_{t\in[0,T]}\norm{B(t)}_{\op(Y;X)}\leq K,
\end{equation}
\item the map 
\begin{equation}
t\in[0,T]\mapsto B(t)x\in X
\end{equation}
is measurable for any $x\in X$.
\end{enumerate}
\item[(3)] The map $f:[0,T]\longrightarrow X$ is bounded and measurable.
\item[(4)] Setting $Q(t)=B(t)B(t)^\star$ for all $t\in [0,T]$. The trace of the operator $Q(t,s)$, defined in \eqref{qtscov}, is finite for every $0\leq s <t\leq T$.
\end{enumerate}
\end{ip}
By the Uniform Boundedness Theorem, there exists $N>0$ such that 
\begin{equation}
\norm{U(t,s)}_{\op(X)}\leq N,\ \ \ (s,t)\in\overline{\Delta}.
\end{equation}
 
We recall the main properties of the Cameron Martin spaces in the Hilbert setting.
\begin{prop}\label{teoria}
Let $\mu=\Ng_ {0,Q}$ be the Gaussian measure with mean $m$ and covariance $Q$. The Cameron-Martin space $\cm$ of $\mu$ is the subset $Q^{\frac{1}{2}}(X)$ endowed with the scalar product $$\iprod{h}{k}_{\cm}=\ix{Q^{-\frac{1}{2}} h}{Q^{-\frac{1}{2}} k},\ \ \ h,k\in Q^{\frac{1}{2}}(X),$$
where $\Qum$ is the pseudo-inverse of $Q$. 

For all $h\in \cm$, there exists $\rwhat{h}\in L^p(X,\mu)$, for every $p\in[1,\infty)$  with 
\begin{equation}\label{stimahcap}
\norm{\,y\longmapsto\rwhat{h}(y)\,}_{L^p(X,\mu)}\leq c_p \norm{h}_{Q^{\frac{1}{2}}(X)}
\end{equation} such that the Cameron-Martin formula $$\Ng_{h,Q}(dy)=\exp\biggl(-\frac{1}{2}\norm{h}^2_{Q^{\frac{1}{2}}(X)}+\rwhat{h}(y)\biggr)\Ng_{0,Q}(dy)$$ holds. 

Let $(e_k)_{k\in\Nat}\subseteq X$ be an orthonormal basis of $X$ consisting of eigenvectors of $Q$, $Qe_k=\la_ke_k$ for all $k\in\Nat$. We have 

\begin{equation}\label{cappucci}
\rwhat{h}(y)=\sum_{\substack{k\in\Nat,\\ \la_k\neq0}}y_k \Bigl(Q^{-\frac{1}{2}}h\Bigr)_k\la_k^{-\frac{1}{2}}=\sum_{\substack{k\in\Nat,\\ \la_k\neq0}}y_k\, h_k\,\la_k^{-1}
\end{equation}
where $y_k=\ix{y}{e_k}$ for any $y\in X$. We remark that the series defined in \eqref{cappucci} converges in $L^p(X,\Ng_{0,Q})$ for any $p\in[1,\infty)$ and it does not converges pointwise if and only if $h=0$. 
\end{prop}

\begin{dfn}
In our non autonomous setting, for every $(t,s)\in\Delta$ we denote by $\cm_{t,s}$ the Cameron-Martin space of the measure $\Ng_{0,Q(t,s)}$.
Moreover, for every $t\in[0,T]$ we denote by $\mathrm{H}_t$ the space $\Qp{t}(X)$ endowed with the scalar product $$\iprod{h}{k}_{\mathrm{H}_t}=\ix{\Qm{t}h}{\Qm{t}k},\ \ \ h,k\in \ac_t,$$
where $\Qm{t}$ is the pseudo-inverse of $\Qp{t}$.
\end{dfn}

\section[]{Smoothing properties of $\pst$}

If Hypothesis \ref{1} holds, the evolution family  $$P_{s,t}\ph(x)=\int_X\ph(y+m^x(t,s))\,\Ng_{0,Q(t,s)}(dy),\ \ \ x\in X,\ (s,t)\in \Delta,\ \ph\in C_b(X),$$ is well defined. 
In \cite{cerlun} it was proved that $P_{s,t}$ maps $C^1_b(X)$ into itself, and $$\nabla(P_{s,t}\ph)(x)=\uts^\star P_{s,t}\nabla\ph(x),\ \ (s,t)\in\Delta,\ x\in X,\ \ph\in C^1_b(X),$$ so that $$\sup_{x\in X}\norm{D(P_{s,t}\ph)(x)}_{X^\star}\leq \norm{\uts}_{\op(X)}\norm{D\ph}_\infty\leq M\norm{\ph}_{C^1_b(X)},\ \ (s,t)\in\Delta,\ \ph\in C^1_b(X).$$
More generally it was proved that $\pst$ maps $C^k_b(X)$ into itself for every $k\in\Nat$.
Here we extend this results working along the directions of a suitable subspace $E$. 

\begin{lem}\label{regck} We assume that Hypothesis \ref{1} holds and that there exists a normed space $(\ets,\norm{\, \cdot\, }_{\ets})$ continuously embedded in $X$ such that for some $(s,t)\in\Delta$ 
\begin{equation}
\label{17} U(t,s)(\ets)\subseteq\ets\,\ \ U(t,s)_{|_{\ets}}\in\op(\ets)\ \mbox{and}\ \exists\ M>0\  \mbox{such that}\ \norm{U(t,s)_{|_{\ets}}}_{\op(\ets)}\leq M.
\end{equation}
Then $P_{s,t}$ maps $C^k_{\ets}(X)$ into itself for every $k\in\Nat$ and 
\begin{equation}\label{dk}
D^k_{\ets}(\pst\ph)(x)(h_1,...,h_k)=\pst \bigl(D_{\ets}^k\ph(\cdot)(\uts h_1,...,\uts h_k)\bigr)(x),
\end{equation}
for any $x\in X$ and $h_1,...,\ h_k\in \ets$. In particular for every $\ph\in C^k_b(X)$ 
\begin{align}
\abs{D^k_{\ets}\pst\ph(x)(h_1,...,h_k)}&\leq\pst\Bigl(\abs{D^k_{\ets}\ph(\cdot)(\uts h_1,...,\uts h_k)}\Bigr), \\
\label{stimadk}\sup_{x\in X}\norm{D_{\ets}^k(P_{s,t}\ph)}_{\op^{k}(\ets)}&\leq \norm{\uts}^k_{\op(\ets)}\sup_{x\in X}\norm{D^k_{\ets}\ph}_{\op^k(\ets)}\leq M^k\sup_{x\in X}\norm{D^k_{\ets}\ph}_{\op^k(\ets)}.
\end{align}

If $\ph\in C^\al_{\ets}(X)$, where $\al=k+\si$, $k\in\Nat\cup\{0\}$ and $\si\in(0,1)$, then $\pst\ph\in C^\al_{\ets}(X)$ and
\begin{equation}\label{sale}
[D_{\ets}^k(\pst\ph)]_{C^\si_{\ets}(X)}\leq\norm{\uts}^{k}_{\op(\ets)}[D^k_{\ets}\ph]_{C^\si_{\ets}(X)}\leq M^{k}[D^k\ph]_{C^\si_{\ets}(X)}.
\end{equation}
If $\ph\in Z^k_{\ets}(X)$, where $k\in\Nat$, then $\pst\ph\in Z^k_{\ets}(X)$ and
\begin{equation}\label{pepe}
[D_{\ets}^{k-1}(\pst\ph)]_{Z^k_{\ets}(X)}\leq\norm{\uts}^{k-1}_{\op(\ets)}[D_{\ets}^k\ph]_{Z^k_{\ets}(X)}\leq M^{k-1}[D_{\ets}^k\ph]_{Z^k_{\ets}(X)}.
\end{equation}
\end{lem}

\begin{oss} By abuse of language, we will often write $\uts$ instead of $\uts_{|_{\ets}}$.
\end{oss}

\begin{proof} 
We start proving \eqref{dk} and \eqref{stimadk} by recurrence over $k$. If $k=1$, let $x\in X$, $h\in \ets$ and $\ep>0$, we have
\begin{align}
\frac{\pst\ph(x+\ep h)-\pst\ph(x)}{\ep}\leq\int_X\frac{\abs{\ph\bigl(y+m^{x+\ep h}(t,s)\bigl)-\ph\bigl(y+m^x(t,s)\bigr)}}{\ep}\,\Ng_{0,Q(t,s)}(dy).
\end{align}
Since {\small $$\frac{\abs{\ph\bigl(y+m^{x+\ep h}(t,s)\bigr)-\ph\bigl(y+m^x(t,s)\bigr)}}{\ep}\leq \norm{D_{\ets}\ph}_\infty \norm{U(t,s)h}_{\ets}\leq\norm{D_{\ets}\ph}_\infty \norm{U(t,s)}_{\op(\ets)}\norm{h}_{\ets},$$} by the Dominated Convergence Theorem, as $\ep\rightarrow 0^+$, we get
\begin{equation}
D_{\ets}\pst\ph(x)(h)=\pst D_{\ets}\ph(\cdot)(U(t,s)h).
\end{equation}
Moreover
\begin{align}
\abs{D_{\ets}\pst\ph(x)(h)}&\leq \int_X \abs{D_{\ets}\ph(y+m^x(t,s))(U(t,s)h)}\, \Ng_{0,Q(t,s)}(dy)\nonumber\\
&\leq\sup_{x\in X}\norm{D_{\ets}\ph}_{\ets^\star}\norm{U(t,s)}_{\op(\ets)}\norm{h}_{\ets},\nonumber
\end{align}
and
\begin{align}
\norm{D_{\ets}\pst\ph(x)}_{\ets^\star}\leq \norm{D_{\ets}\ph}_{\infty}\norm{U(t,s)}_{\op(\ets)}.
\end{align}
Hence
\begin{align}
\sup_{x\in X}\norm{D_{\ets}\pst\ph}_{\ets^\star}\leq \sup_{x\in X}\norm{D_{\ets}\ph}_{\ets^\star}\norm{U(t,s)}_{\op(\ets)}\leq M \norm{\ph}_{C_{\ets}^1(X)}.
\end{align}
We assume now $\ph\in C^k_{\ets}(X)$ and that \eqref{dk} and \eqref{stimadk} hold. 
 Let $x\in X$, $h_1,...,h_{k+1}\in \ets$ and $\ep>0$. Then we have
\begin{align}
&\frac{1}{\ep}\Bigl(\pst\bigl(D_{\ets}^k\ph(\cdot)(\uts h_1,...,\uts h_k)\bigr)(x+\ep h_{k+1})-\pst\bigl(D_{\ets}^k\ph(\cdot)(\uts h_1,...,\uts h_k)\bigr)(x)\Bigr)\nonumber \\ &\leq\frac{1}{\ep}\int_X \Bigl|D_{\ets}^k\ph\bigl(y+m^{x+\ep h_{k+1}}(t,s)\bigl)(U(t,s)h_1,...,U(t,s)h_k)\nonumber\\
&-D_{\ets}^k\ph\bigl(y+m^x(t,s)\bigr)(U(t,s)h_1,...,U(t,s)h_k)\Bigr|\,\Ng_{0,Q(t,s)}(dy).\nonumber
\end{align}
Since 
{\small \begin{align}
&\frac{1}{\ep}\abs{D_{\ets}^k\ph\bigl(y+m^{x+\ep h_{k+1}}(t,s)\bigl)(U(t,s)h_1,...,U(t,s)h_k)-D_{\ets}^k\ph\bigl(y+m^x(t,s)\bigr)(U(t,s)h_1,...,U(t,s)h_k)}\nonumber\\
&\leq \frac{1}{\ep}\norm{D^k_{\ets}\ph\bigl(y+m^{x+\ep h_{k+1}}(t,s)\bigr)-D^k_{\ets}\ph\bigl(y+m^x(t,s)\bigr)}_{\op^{k+1}(\ets)} \prod_{i=1}^{k}\norm{U(t,s)h_{i}}_{\ets}\nonumber\\
&\leq\sup_{x\in X}\norm{D^{k+1}_{\ets}\ph}_{\op^{k+1}(\ets)} \norm{U(t,s)}^{k+1}_{\op(\ets)}\prod_{i=1}^{k+1}\norm{h_i}_{\ets}\nonumber
\end{align}}
by the Dominated Convergence Theorem, as $\ep\rightarrow 0^+$, we get
\begin{equation}
D^{k+1}_{\ets}(\pst\ph)(x)(h_1,...,h_{k+1})=\pst \bigl(D_{\ets}^{k+1}\ph(\cdot)(\uts h_1,...,\uts h_{k+1})\bigr)(x).
\end{equation}
Moreover
{\small\begin{align}
\abs{D^{k+1}_{\ets}(\pst\ph)(x)(h_1,...,h_{k+1})}&\leq \int_X \abs{D^{k+1}_{\ets}\ph(y+m^x(t,s))(\uts h_1,...,\uts h_{k+1})}\, \Ng_{0,Q(t,s)}(dy)\nonumber\\
&\leq\sup_{x\in X}\norm{D^{k+1}_{\ets}\ph}_{\op^{k+1}(\ets)}\norm{U(t,s)}^{k+1}_{\op(\ets)}\prod_{i=1}^{k+1}\norm{h_i}_{\ets},\nonumber
\end{align}}
so that
\begin{equation}
\norm{D^{k+1}_{\ets}(\pst\ph)(x)}_{\op^k(\ets)}\leq\sup_{x\in X}\norm{D^{k+1}_{\ets}\ph}_{\op^{k+1}(\ets)}\norm{U(t,s)}^{k+1}_{\op(\ets)}.
\end{equation}
Hence
\begin{equation}
\sup_{x\in X}\norm{D^{k+1}_{\ets}(\pst\ph)}_{\op^{k+1}(\ets)}\leq\norm{U(t,s)}^{k+1}_{\op(\ets)}\sup_{x\in X}\norm{D^{k+1}_{\ets}\ph}_{\op^{k+1}(\ets)}\leq M^{k+1}\norm{\ph}_{C_{\ets}^{k+1}(X)}.
\end{equation}

%
%
%
%
Finally \eqref{sale}, \eqref{pepe} are consequences of \eqref{dk} and of the definition of $\pst$.
\end{proof}

\begin{oss}\label{ossgc}
If $E$ and $F$ are Banach spaces continuously embedded in $X$ such that $U(t,s)(E)\subseteq F$, then $U(t,s)_{|_E}\in\op(E;F)$. Indeed, if $x_n\xrightarrow[n\rightarrow+\infty]{E}x$ and $U(t,s)_{|_E}x_n\xrightarrow[n\rightarrow+\infty]{F}y$, then $x_n\xrightarrow[n\rightarrow+\infty]{X}x$ and $U(t,s)x_n\xrightarrow[n\rightarrow+\infty]{X}y$, since the embeddings of $E$ in $X$ and of $F$ in $X$ are continuous. Since $U(t,s)\in\op(X)$ then $y=U(t,s)x$ and by the Closed Graph Theorem $U(t,s)_{|_E}\in\op(E,F)$.
\end{oss}

\begin{lem}\label{regck2}
We assume that Hypothesis \ref{1} holds and that \eqref{17} holds for every $(s,t)\in\Delta$. Then, for every fixed $t\in[0,T]$, $\ph\in C^k_E(X)$ and $h_1,...,\ h_k\in X$, the mapping
\begin{equation}\label{dek1}
(s,x)\in [0,t]\times X\longmapsto D_{\ets}^k(\pst\ph)(x)(h_1,...,h_k)\in\R,
\end{equation}
is continuous. 

If in addition the mapping $(s,t)\in\Delta\longmapsto\uts\in\op(X)$ is continuous, then the function
\begin{equation}\label{dek2}
(s,t,x)\in \overline{\Delta}\times X\longmapsto D_{\ets}^k(\pst\ph)(x)(h_1,...,h_k)\in\R,
\end{equation}
is continuous.\end{lem}

\begin{proof}
The proofs of these continuity properties are analogous to \cite[Lemma 2.3]{cerlun}.\end{proof}

\begin{cor}\label{u0} We assume that Hypothesis \ref{1} holds and that there exists a normed space $(\ets,\norm{\, \cdot\, }_{\ets})$ continuously embedded in $X$ such that \eqref{17} holds for every $(s,t)\in\Delta$. Then, for every $\al>0$ and for every $\ph\in C^\al_{\ets}(X)$, $t\in(0,T]$, the function $$(s,x)\in[0,t]\times X\longmapsto u_0(s,x):=\pst\ph(x)\in\R,$$ belongs to $C^{0,\al}_{\ets}\bigl([0,t]\times X\bigr)$, and there exists $C=C(\al, T)>0$ such that
\begin{equation}
\norm{u_0}_{C^{0,\al}_{\ets}([0,t]\times X)}\leq C\norm{\ph}_{C^\al_{\ets}(X)}.
\end{equation}
Similarly, if $\ph\in Z^k_{\ets}(X)$ for some $k\in\Nat$, then the function $u_0$ belongs to $Z^{0,k}_{\ets}\bigl([0,t]\times X\bigr)$, and there exists $C=C(k, T)>0$ such that
\begin{equation}
\norm{u_0}_{Z^{0,k}_{\ets}([0,t]\times X)}\leq C\norm{\ph}_{Z^k_{\ets}(X)}.
\end{equation}
\end{cor}

\begin{thm}\label{regcm}
Assuming that Hypothesis \ref{1} holds and there exists a normed space $(\ets,\norm{\, \cdot\, }_{\ets})$ continuously embedded in $X$ such that \eqref{16} holds, 
 $\displaystyle{\pst\ph\in\bigcap_{k\in\Nat} C^k_{\ets}(X)}$ for every $\ph\in C_b(X)$ and
\begin{equation}\label{grad}
D_{\ets}(\pst\ph)(x)(h)=\int_X\ph\bigl(y+m^x(t,s)\bigr)\,\rwhat{ U(t,s)h}(y)\,\Ng_{0,Q(t,s)}(dy),\ \ h\in \ets,
\end{equation}
where $\,\widehat{\cdot}\,$ is with respect to the Gaussian measure $\Ng_{0,Q(t,s)}$ (see Proposition \ref{teoria}).

Moreover, there exists $C>0$ such that
\begin{equation}
\sup_{x\in X}\norm{D_{\ets}(\pst\ph)(x)}_{\ets^\star}\leq C\norm{\Lats}_{\op(\ets; X)}\norm{\ph}_\infty.
\end{equation}
For $n\geq2$
{\small \begin{equation} \label{derivate}
D^n_{\ets}(\pst\ph)(x)(h_1,...,h_n)=\int_X\ph\bigl(y+m^x(t,s)\bigr)\, I_n(t,s)(y)(h_1,...,h_n)\, \Ng_{0,Q(t,s)}(dy),\ \ h_1,...,h_n\in \ets,
\end{equation}}
where
{\small\begin{align}
I_n(t,s)(y)(h_1,...,h_n)&:=\prod_{i=1}^n\rwhat{ U(t,s)h_i}(y)\nonumber \\ 
&+\sum_{s=1}^{r_n}(-1)^s\sum_{\substack{i_1,...i_{2s},\\ i_{2k-1}<i_{2k},\\ i_{2k-1}<i_{2k+1}}}^n\prod_{i=1}^s\ix{\Lats h_{i_{2k-1}}}{\Lats h_{i_{2k}}}\prod_{\substack{i_m=1,\\ i_m\neq i_1,...,i_{2s}}}^n\rwhat{ U(t,s)h_{i_m}}(y)\end{align}}
and
\begin{equation}
r_n=
\begin{cases}
\frac{n}{2}\ \ \mbox{if}\ n\ \mbox{is even},\\
\frac{n-1}{2}\ \ \mbox{if}\ n\ \mbox{is odd}. \nonumber
\end{cases}
\end{equation}
It follows that for every $n\geq 2$ there exists $C_n>0$ such that
\begin{equation}\label{stimadn}
\sup_{x\in X}\norm{D_{\ets}^n(\pst\ph)(x)}_{\op^n(\ets)}\leq C_n\norm{\Lats}^n_{\op(\ets; X)}\norm{\ph}_\infty.
\end{equation}
\end{thm}
\begin{proof}
\leavevmode
\begin{enumerate}
\item[Step 1.] We prove that $\pst\ph(x)\in C_b(X)$ for every $\ph\in C_b(X)$. 
For $x, x_0\in X$ we have
\begin{align}
\abs{\pst\ph(x)-\pst\ph(x_0)}\leq\int_X\abs{\ph\bigl(y+m^x(t,s)\bigr)-\ph\bigl(y+m^{x_0}(t,s)\bigr)}\,\Ng_{0,Q(t,s)}(dy).
\end{align}
Since $$\abs{\ph\bigl(y+m^x(t,s)\bigr)-\ph\bigl(y+m^{x_0}(t,s)\bigr)}\leq2 \norm{\ph}_\infty$$ the statement follows letting $x\rightarrow x_0$, by the Dominated Convergence theorem . 

\item[Step 2.] We prove that $\pst\ph(x)\in C^1_{\ets}(X)$ for every $\ph\in C_b(X)$ and that \eqref{grad} holds.\\
For every $\ep\in (0,1)$, $x\in X$ and $h\in \ets$ we have
\begin{align} 
\frac{\pst\ph(x+\ep h)-\pst\ph(x)}{\ep}=&\frac{1}{\ep}\biggl(\int_X\ph(y+m^x(t,s))\,\Ng_{\ep U(t,s)h,Q(t,s)}(dy)\nonumber \\
&-\int_X\ph(y+m^x(t,s))\,\Ng_{0,Q(t,s)}(dy)\biggr). \nonumber
\end{align} 
Since $\ep\,U(t,s)h\in\cm_{t,s}$, thanks to the Cameron-Martin formula we get 
\begin{align}
\Ng_{\ep U(t,s)h,Q(t,s)}(dy)&=  \exp\biggl(-\frac{1}{2}\norm{\ep U(t,s)h}^2_{\cm_{t,s}}+\rwhat{\ep U(t,s)h}(y)\biggr)\Ng_{0,Q(t,s)}(dy)\nonumber\\
&=\exp\biggl(-\frac{1}{2}\ep^2\norm{\Lats h}^2_X+\ep\rwhat{U(t,s)h}(y)\biggr)\Ng_{0,Q(t,s)}(dy). \nonumber
\end{align}
Therefore, setting 
\begin{equation}
f_\ep(y)=-\frac{1}{2}\ep\norm{\Lats h}^2_X+\rwhat{ U(t,s)h}(y),
\end{equation}
we get
\begin{equation}
\frac{\pst\ph(x+\ep h)-\pst\ph(x)}{\ep}=\int_X\frac{\exp(\ep f_\ep(y))-1}{\ep}\,\ph\bigl(y+m^x(t,s)\bigr)\,\Ng_{0,Q(t,s)}(dy).
\end{equation}
Now 
\begin{align}
&\lim_{\ep\rightarrow 0} f_\ep(y)=\rwhat{U(t,s)h}(y)\ \ \mbox{for a.e.}\ y \nonumber\\
&\abs{\frac{\exp\Bigl(\ep f_\ep(y)\Bigr)-1}{\ep}\,\ph\bigl(y+m^x(t,s)\bigr)}\leq C\abs{f_\ep(y)}\Bigl(\exp\abs{ f_\ep(y)}+1\Bigr)\norm{\ph}_\infty\nonumber \\
&y\mapsto \rwhat{U(t,s)h}(y)\exp\biggl(\rwhat{ U(t,s)h}(y)\biggr)\in L^1(X,\Ng_{0,Q(t,s)}).\nonumber
\end{align}
Hence by the Dominated Convergence theorem we obtain 
\begin{equation}\label{dom}
\lim_{\ep\rightarrow 0}\frac{\pst\ph(x+\ep h)-\pst\ph(x)}{\ep}=\int_X\ph(y+m^x(t,s))\rwhat{U(t,s)h}(y)\,\Ng_{0,Q(t,s)}(dy).
\end{equation}
Moreover, by \eqref{stimahcap} with $p=1$ we get
\begin{align}
\abs{\int_X\ph(y+m^x(t,s))\rwhat{U(t,s)h}(y)\,\Ng_{0,Q(t,s)}(dy)}&\leq \norm{\ph}_{\infty}\norm{\rwhat{U(t,s)h}(\cdot)}_{L^1(X,\Ng_{0,Q(t,s)})}\nonumber\\
 &\leq\norm{\ph}_{\infty}c_1\norm{\Lats}_{\op(\ets; X)}\norm{h}_{\ets}, \nonumber
\end{align}
so that $\pst\ph$ is $\ets$-G\^{a}teaux differentiable at $x$ and  $$D_{\ets}^G(\pst\ph)(x)=\int_X\ph(y+m^x(t,s))\rwhat{U(t,s)h}(y)\,\Ng_{0,Q(t,s)}(dy).$$
To conclude we prove that $D_{\ets}^G(\pst\ph):X\longrightarrow \ets^\star$ is continuous. Indeed for $x, x_0\in X$ we apply \eqref{stimahcap} with $p=2$ and we have
\begin{align}
&\abs{D_{\ets}^G(\pst\ph)(x)(h)-D_{\ets}^G(\pst\ph)(x_0)(h)}\nonumber\\
&\leq\int_X\abs{\ph(y+m^{x}(t,s))-\ph(y+m^{x_0}(t,s))}\abs{\rwhat{U(t,s)h}(y)}\,\Ng(0,Q(t,s))(dy)\nonumber\\
&\leq c_2\norm{\ph(y+m^{x}(t,s))-\ph(y+m^{x_0}(t,s))}_{L^2(X,\Ng_{0,Q(t,s)})}\norm{\Lats}_{\op(\ets; X)}\norm{h}_{\ets}.\nonumber
\end{align}
Hence $$\norm{D_{\ets}^G(\pst\ph)(x)-D_{\ets}^G(\pst\ph)(x_0)}_{\ets^\star}\leq\overline{c}\norm{\ph(y+m^{x}(t,s))-\ph(y+m^{x_0}(t,s))}_{L^2(X,\Ng_{0,Q(t,s)})}$$
and since $\abs{\ph\bigl(y+m^x(t,s)\bigr)-\ph\bigl(y+m^{x_0}(t,s)\bigr)}\leq2 \norm{\ph}_\infty$, the statement follows by the Dominated Convergence Theorem letting $x\rightarrow x_0$.


\item[Step 3.] We prove that $\pst\ph\in C^2_{\ets}(X)$. We show first that for every  $h\in \ets$, the mapping $$D_{\ets}(\pst\ph)(\cdot)(h): X\longrightarrow\R$$ is  $\ets$-G\^{a}teaux differentiable. Let $\widetilde{h}\in \ets$, we set
\begin{equation*}
f_\ep(y)=-\frac{1}{2}\ep\norm{\Lats \widetilde{h}}^2_X+\rwhat{U(t,s)\widetilde{h}}(y),
\end{equation*}
and due to \eqref{grad} and using again the Cameron-Martin formula, for every $\ep>0$ we have
\begin{align}
&\frac{D_{\ets}(\pst\ph)(x+\ep \widetilde{h})(h)-D_{\ets}(\pst\ph)(x)(h)}{\ep}\nonumber \\
&=\frac{1}{\ep}\biggl[\int_X\ph\bigl(y+m^x(t,s)\bigr)\,\rwhat{U(t,s)h}\bigl(y-\ep \uts \widetilde{h}\bigr)\,\Ng_{\ep\uts \widetilde{h},Q(t,s)}(dy)\nonumber \\
&-\int_X\ph\bigl(y+m^x(t,s)\bigr)\,\rwhat{U(t,s)h}(y)\,\Ng_{0,Q(t,s)}(dy)\biggr]\nonumber\\
&=\frac{1}{\ep}\int_X\ph\bigl(y+m^x(t,s)\bigr)\rwhat{U(t,s)h}(y)\bigl(\exp(\ep f_\ep(x))-1\bigr)\Ng_{0,Q(t,s)}(dy)\nonumber\\
&-\int_X\ph(y+m^x(t,s))\ix{\Lats h}{\Lats\widetilde{h}}\exp(\ep f_\ep(y))\Ng_{0,Q(t,s)}(dy).\nonumber
\end{align}
Proceeding as in Step 2, we obtain
\begin{align}\label{conv2}
&\lim_{\ep\rightarrow 0}\frac{D_{\ets}(\pst\ph)(x+\ep \widetilde{h})h-D_{\ets}(\pst\ph)(x)h}{\ep}\nonumber \\
&=\int_X\ph\bigl(y+m^x(t,s)\bigr)\,\rwhat{U(t,s)h}(y)\,\rwhat{U(t,s)\widetilde{h}}(y)\,\Ng_{0,Q(t,s)}(dy)\nonumber\\
 &-\ix{\Lats h}{\Lats\widetilde{h}}\pst\ph(x).
\end{align}
The right-hand side of \ref{conv2} is the G\^{a}teaux derivative of $D_{\ets}\pst\ph(\cdot)h$ at $x$. In the same way we did in Step 2, we can show that   the mapping $D_{\ets}^G(D_{\ets}(\pst\ph)(\cdot)h):X\longrightarrow \ets^\star$ is continuous, so that we conclude that $\pst\ph\in C^2_{\ets}(X)$.
\item[Step 4.] We prove that $\pst\ph\in C^n_{\ets}(X)$ for every $n\in\Nat$ and that \eqref{derivate} holds.
We proceed by recurrence, so we assume that $\pst\ph\in C^{n-1}_{\ets}(X)$ and that formula \eqref{derivate} holds for $D^{n-1}_{\ets}(\pst\ph)$, for some $n\geq 3$.

We first show that $D^{n-1}_{\ets}(\pst)\ph$ is $\ets$- G\^ateaux differentiable. For $x\in X$, $h_1,...,h_{n-1},h_n\in \ets$ and $\ep>0$ we set
 \begin{equation*}
f_\ep(y)=-\frac{1}{2}\ep\norm{\Lats h_n}^2_X+\rwhat{U(t,s)h_n}(y).
\end{equation*}
and due to the Cameron-Martin Formula we have
 \begin{align}
 &D^{n-1}_{\ets}(\pst\ph)(x+\ep h_n)(h_1,...,h_{n-1})\nonumber \\
 &=\int_X\ph(m^x(t,s)+y) I_{n-1}(t,s)(y-\ep \uts h_n) \exp(\ep f_\ep(y))\nonumber\Ng_{0,Q(t,s)}(dy).
  \end{align}
Moreover we have 
\begin{align}
 I_{n-1}&(t,s)(y-\ep \uts h_n)(h_1,...,h_{n-1})=I_{n-1}(t,s)(y)(h_1,...,h_{n-1}) \nonumber \\
 &-\ep\sum_{i=1}^{n-1}\ix{\Lats h_i}{\Lats h_n}\prod_{\substack{j=1\\j\neq i}}^{n-1}\rwhat{U(t,s)h_j}(y)\nonumber \\
 &-\ep\sum_{s=1}^{r_{n-1}}(-1)^s\sum_{\substack{i_1,...i_{2s},\\ i_{2k-1}<i_{2k},\\ i_{2k-1}<i_{2k+1}}}^n\prod_{i=1}^s\ix{\Lats h_{i_{2k-1}}}{\Lats h_{i_{2k}}}\nonumber \\
&\times\sum_{\substack{i_m=1,\\ i_m\neq i_1,...,i_{2s}}}^{n-1}\ix{\Lats h_{i_m}}{\Lats h_n}\prod_{\substack{i_j=1\\i_j\neq i_m,i_1,...,i_{2s}}}^{n-1}\rwhat{U(t,s)h_{i_j}}(y)+O(\ep^2).\nonumber 
\end{align}
In the same way we did in Step 2 and in Step 3, by the Dominated Convergence theorem we have 
\begin{align}
&\lim_{\ep\rightarrow 0}\frac{D^{n-1}_{\ets}(\pst\ph)(x+\ep h_n)(h_1,...,h_{n-1})-D^{n-1}_{\ets}(\pst\ph)(x)(h_1,...,h_{n-1})}{\ep}=\nonumber \\
&=\int_X \ph(m^x(t,s)+y) I_{n-1}(t,s)(y)(h_1,...,h_{n-1})\,\rwhat{U(t,s)h_n}(y)\,\Ng_{0,Q(t,s)}(dy)+\nonumber \\
&-\int_X \ph(m^x(t,s)+y) \Biggl[\sum_{i=1}^{n-1}\ix{\Lats h_i}{\Lats h_n}\prod_{\substack{j=1\\j\neq i}}^{n-1}\rwhat{U(t,s)h_j}(y)+\nonumber \\
&+\sum_{s=1}^{r_{n-1}}(-1)^s\sum_{\substack{i_1,...i_{2s},\\ i_{2k-1}<i_{2k},\\ i_{2k-1}<i_{2k+1}}}^n\prod_{i=1}^s\ix{\Lats h_{i_{2k-1}}}{\Lats h_{i_{2k}}}\times\nonumber \\
&\times\sum_{\substack{i_m=1,\\ i_m\neq i_1,...,i_{2s}}}^{n-1}\ix{\Lats h_{i_m}}{\Lats h_n}\prod_{\substack{i_j=1\\i_j\neq i_m,i_1,...,i_{2s}}}^{n-1}\rwhat{U(t,s)h_{i_j}}(y)\Biggr]\,\Ng_{0,Q(t,s)}(dy). \nonumber
\end{align}
It is easy to show that the right hand side of the equation above coincides with the expression of $D^n_{\ets}(\pst\ph)(h_1,...,h_n)$ given in \eqref{derivate}. The same arguments of Step 2 and Step 3 imply that $\pst\ph\in C^n_{\ets}(X)$.
\end{enumerate}
\end{proof}

\begin{oss}
We assume that the hypotheses of Theorem \ref{regcm} hold and that $E$ is a Banach space. Then  $U(t,s)_{|_{\ets}}\in\op(\ets,\cm_{t,s})$ if and only if $\Lats\in\op(E;X)$. 
Indeed, $\norm{U(t,s)}_{\op(E,\cm_{t,s})}=\norm{\Lats}_{\op(E,X)}$ and by Remark \ref{ossgc} $U(t,s)$ belongs to $\op(E,\cm_{t,s})$.
\end{oss}

\begin{oss}If $\ets=X$ in Theorem \ref{regcm}, then $\pst$ is strong Feller as proved in \cite{cerlun, MR3466574}.
\end{oss}
\begin{oss}\label{3bis}
Similarly to \cite{cerlun}, fixing $(s,t)\in\Delta$ we define the operator $L:L^2\bigl((s,t); Y\bigr)\longrightarrow X$ given by 
\begin{equation}\label{Lop}
Ly:=\int_s^t\utz B(\si)y(\si)\,d\si,\ \ \ y\in L^2\bigl((s,t); Y\bigr).
\end{equation} 
The adjoint operator $L^\star:X\longrightarrow L^2\bigl((s,t); Y\bigr)$ satisfies $\ix{Ly}{x}=\iprod{y}{L^\star x}_{L^2((s,t); Y)}$ for every $x\in X$ and $y\in L^2\bigl((s,t); Y\bigr)$, which means $$\int_s^t\ix{\utz B(\si)y(\si)}{x}\,d\si=\int_s^t\iprod{y(\si)}{(L^\star x)(\si)}_X\,d\si,\ \ \ x\in X,y\in L^2\bigl((s,t); Y\bigr),$$ so that $$(L^\star x)(\si)=B^\star(\si) U^\star(t,\si)x,\ \ x\in X,\ \mbox{a.e}\ \si\in(s,t)$$ and we get $$LL^\star x=\int_s^t\utz B(\si)B^\star(\si) U^\star(t,\si)x \,d\si=Q(t,s)x,\ \ x\in X.$$
Therefore, by the general theory of linear operators in Hilbert spaces (e.g., \cite[Cor. B3]{MR3236753}), we get
\begin{equation}
\begin{cases}
\re{L}=\re{\Qp{t,s}} \\
~\\
\norm{\Qm{t,s} x}_X=\norm{L^{-1}x}_{L^2((s,t); Y)},\ \ \ x\in \cm_{t,s}.
\end{cases}
\end{equation}
Since in general $L$ is not invertible, $L^{-1}$ is meant as the pseudo-inverse of $L$. Hence 
\begin{equation}\label{Lpseudo}
\norm{L^{-1}x}_{L^2((s,t); Y)}=\min\Bigl\{\norm{y}_{L^2((s,t); Y)}\ :\ Ly=x\Bigr\},\ \ \ x\in \cm_{t,s}.
\end{equation}
The range of $L$ is the set of the traces at time $t$ of the mild solutions of the evolution problems

\begin{equation}
\begin{cases}
u'(r)=A(r)u(r)+B(r)y(r),\ \ s<r<t,\\
u(s)=0
\end{cases}
\end{equation}
where $y$ varies in $L^2((s,t);Y)$. So \eqref{16} may be reformulated requiring that $U(t,s)$ maps $\ets$ in the trace space.
\end{oss}

Now we combine Lemma \ref{regck} and Theorem \ref{regcm} to obtain the following result.

\begin{cor}\label{reg}
We assume that Hypotheses \ref{1} holds and there exists a normed space $(\ets,\norm{\, \cdot\, }_{\ets})$ continuously embedded in $X$ satisfying \eqref{16} and \eqref{17}. Then for every $k,n\in\Nat\cup\{0\}$ such that $k+n\geq 1$, $\ph\in C^k_{\ets}(X)$, $\pst\ph\in C_{\ets}^{k+n}(X)$ and 
\begin{align}\label{dnk}
&D_{\ets}^{k+n}(\pst\ph)(x)(h_1,...,h_{k+n})\nonumber\\
&=\int_X D_{\ets}^k\ph\bigl(\mts+y\bigr)\bigl(\uts h_1,...,\uts h_k\bigr) I_n(t,s)(y)(h_{k+1},...,h_{k+n})\,\Ng_{0,Q(t,s)}(dy).
\end{align}
For every $\al_2\geq\al_1\geq 0$ there exists $C=C(\al_1,\al_2)$ such that
\begin{equation} \label{alstima}
\norm{\pst\ph}_{C^{\al_2}_{\ets}(X)}\leq C\bigl(\norm{\Lats}_{\op(\ets;X)}^{\al_2-\al_1}+1\bigr)\norm{\ph}_{C^{\al_1}_{\ets}(X)},\ \ \ph\in C^{\al_1}_{\ets}(X).\end{equation}
\end{cor}

\begin{proof} 
Formula \eqref{dnk} follows applying \eqref{dk} and \eqref{derivate}. Moreover, by \eqref{stimadk} and \eqref{stimadn} there exists $C=C(k,n)>0$ such that
\begin{equation}
\sup_{x\in X}\norm{D_{\ets}^{k+n}(\pst\ph)(x)}_{\op^{k+n}(\ets)}\leq C\norm{\Lats}^n_{\op(\ets;X)}\norm{\ph}_{C^k_{\ets}(X)},\ \ \ph\in C^k_{\ets}(X),
\end{equation}
and for every $\al\in(0,1)$
\begin{equation}
\norm{D_{\ets}^{k+n}(\pst\ph)}_{C^\al_{\ets}(X,\op^{n+k}(\ets))}\leq C\norm{\Lats}^n_{\op(\ets;X)}\norm{\ph}_{C^k_{\ets}(X)},\ \ \ph\in C^{k+\al}_{\ets}(X).\end{equation}
We point out that if $\al_1=\al_2$, then \eqref{alstima} follows from Lemma \ref{regck}; so we may assume $\al_2>\al_1$. Now if $\al_2-\al_1=n\in\Nat$, we set $\al_1=m+\si$, where $m\in\Nat$ and $\si\in(0,1)$. We have 
\begin{align}
&\norm{\pst\ph}_{C^{\al_2}_{\ets}}=\norm{\pst\ph}_\infty+\sum_{j=1}^{m+n}\sup_{x\in X}\norm{D^j_{\ets} (\pst\ph)(x)}_{\op^j(\ets)}+[D_{\ets}^{m+n}\pst\ph]_{C^\si_{\ets}(X)}\nonumber\\
&\leq \norm{\ph}_\infty+\sum_{j=1}^{m}\sup_{x\in X}\norm{D_{\ets}^j \ph(x)}_{\op^j(\ets)}+\sum_{j=1}^{n}C_j\norm{\Lats}^j_{\op(\ets;X)}\sup_{x\in X}\norm{D_{\ets}^m\ph(x)}_{\op^m(\ets)}+\nonumber\\
&+C\norm{\Lats}^n_{\op(\ets;X)}[D_{\ets}^{m}\ph]_{C^\si_{\ets}(X)}
\leq C(\al_1,\al_2)\bigl(\norm{\Lats}_{\op(\ets;X)}^{\al_2-\al_1}+1\bigr)\norm{\ph}_{C^{\al_1}_{\ets}(X)}.\nonumber
\end{align}
If $\al_2-\al_1\notin\Nat$, we set $\al_2=\al_1+n+\si$ with $n\in\Nat\cup\{0\}$ and $\si\in(0,1)$. We apply the following interpolation inequality 
\begin{equation}\label{interp}
\norm{\psi}_{C^{\al_2}_{\ets}(X)}\leq\norm{\psi}_{C^{\al_1+n}_{\ets}(X)}^{1-\si}\norm{\psi}_{C^{\al_1+n+1}_{\ets}(X)}^\si,\ \ \psi\in C^{\al_1+n+1}_{\ets}(X)
\end{equation}
to $\psi=\pst\ph$. Then, due to \eqref{alstima} with $\al_1$ replaced by $\al_1+n$ and $\al_2$ by $\al_1+n+1$, we have \eqref{alstima} in the general case.
\end{proof}

\begin{oss}
The interpolation inequality \eqref{interp} is proved in \cite[Prop. 2.8]{cerlun} in the case $\ets=X$ with equivalent norms. Anyway the proof of the general case is analogous.
\end{oss}

\begin{cor}
We assume that Hypotheses \ref{1} holds and there exists a normed space $(\ets,\norm{\, \cdot\, }_{\ets})$ continuously embedded in $X$ satisfying \eqref{16} and \eqref{17} for all $(s,t)\in\Delta$. Then for every $j\in\Nat$, $h_1,...,h_j\in \ets$, $t\in [0,T]$ and $\ph\in C_{\ets}(X)$ the mapping $$(s,x)\in [0,t)\times X\longmapsto D^j_{\ets}\pst\ph(x)(h_1,...,h_j)\in\R$$ is continuous.

If in addition the mapping $(s,t)\in\Delta\longmapsto\uts\in\op(X)$ is continuous, then the function
\begin{equation}
(s,t,x)\in \Delta\times X\longmapsto D^j_{\ets}(\pst\ph)(x)(h_1,...,h_j)\in\R,
\end{equation}
is continuous.
\end{cor}

\begin{proof}
The statement about the continuity of the derivatives is a consequence of Lemma \ref{regck} and Theorem \ref{regcm}. Indeed for a fixed $t>0$ and $\ep\in(0,t)$, we have $$\pst\ph=P_{s,t-\ep}P_{t-\ep,t},\ \ 0\leq s< t-\ep\leq T.$$
Since $\psi=P_{t-\ep,t}\ph\in C^k_{\ets}(X)$ by Theorem \ref{regcm}, the function
$$(s,x)\in[0,t-\ep]\times X\longmapsto D_{\ets}^k P_{s,t-\ep}\psi(x)(h_1,...,h_j)=D_{\ets}^k\pst\ph(h)(h_1,...,h_j)$$ is continuous by Lemma \ref{regck}. The proof of the last claim is similar. 
\end{proof}

Finally we state also the last hypothesis that is essential to prove Schauder estimates. 
\begin{ip} \label{3}
\leavevmode
\begin{enumerate}
 \item There exists a normed space $(\ets,\norm{\, \cdot\, }_{\ets})$ continuously embedded in $X$ satisfying \eqref{16} and \eqref{17} for all $(s,t)\in\Delta$.
 \item There exist $C,\theta>0$ such that
\begin{equation}
\norm{\La(t,s)}_{\op(E;X)}\leq \frac{C}{(t-s)^\te},\ \ \mbox{for all}\ (s,t)\in\Delta.
\end{equation}
\end{enumerate}
\end{ip}

We have immediately the following

\begin{cor}\label{blowup}
Let Hypotheses \ref{1} and \ref{3} hold. For every $n\in\Nat$ there exists $K_n>0$ such that
\begin{equation}\label{kn}
\norm{D^n_E\pst\ph(x)}_{\op^n(E)}\leq\frac{K_n}{(t-s)^{n\te}}\norm{\ph}_\infty,\ \ (s,t)\in\Delta,\ \ph\in C_b(X).
\end{equation}
For every $\al\in(0,1)$ and $n\in\Nat$ there exists $K_{n,\al}>0$ such that
\begin{equation}\label{kal}
\norm{D^n_E\pst\ph(x)}_{\op^n(E)}\leq\frac{K_{n,\al}}{(t-s)^{(n-\al)\te}}\norm{\ph}_{C^\al_b(X)},\ \ (s,t)\in\Delta,\ \ph\in C_b(X).
\end{equation}
\end{cor} 
\begin{proof}
Estimates \eqref{kn} and \eqref{kal} follow from Theorem \ref{regcm} and Corollary \ref{reg} taking into account Hypothesis \ref{3}.
\end{proof}


In the following propositions we give some examples of  Hilbert spaces satisfying \eqref{16} and \eqref{17}.  We start with two preliminary results.

\begin{prop}\label{pseudo} Let $X,X_1,X_2$ be Hilbert spaces and let $L_1:X_1\longrightarrow X$, $L_2:X_2\longrightarrow X$ be linear bounded operators. The following statements hold.
\begin{enumerate}
\item $\re{L_1}\subseteq\re{L_2}$ if and only if there exists a constant $C>0$ such that 
\begin{equation}\label{cstar}
\norm{L_1^\star x}_{X_1}\leq C\norm{L_2^\star x}_{X_2},\ \ x\in X.
\end{equation}
In this case $\norm{L^{-1}_2L_1}_{\op(X_1,X_2)}\leq C$; more precisely
\begin{equation}
\norm{L^{-1}_2L_1}_{\op(X_1,X_2)}=\inf\{C>0\ s.t.\ \eqref{cstar}\ holds\}
\end{equation}
\item If $\norm{L_1^\star x}_{X_1}=\norm{L_2^\star x}_{X_2}$ for every $x\in X$ then $\re{L_1}=\re{L_2}$ and \\
$\norm{L^{-1}_1 x}_{X_1}=\norm{L^{-1}_2 x}_{X_2}$ for every $x\in X$.
\end{enumerate}
\end{prop}
\begin{proof}
See Proposition B.1 in Appendix B in \cite[pag. 429]{MR3236753}.
\end{proof}

\begin{lem} \label{cmQ}
Under Hypothesis \ref{1}, the following properties hold.
\begin{enumerate}
\item the mapping $s\longmapsto Q(t,s)$ is Lipschitz continuous with values in $\op(X)$ and it is decreasing, namely $$\ix{Q(t,s_1)x}{x}\geq\ix{Q(t,s_2)x}{x} \ \ \ \mbox{for any}\ \ \  0\leq s_1\leq s_2 < t\leq T.$$
\item $\cm_{t,s_2}\subseteq \cm_{t,s_1}$ for every $0\leq s_1<s_2 < t\leq T$ and the norm of the embedding is $\leq$ 1.
\item $\Ker(Q(t,s_1))\subseteq \Ker (Q(t,s_2))\subseteq \Ker(Q(t)) \ \mbox{for any}\ \   0\leq s_1\leq s_2 < t\leq T$.
\end{enumerate}
\end{lem}
\begin{proof}
\leavevmode
\begin{enumerate}
\item Let $0\leq s_1\leq s_2 < t\leq T$ and $x\in X$. We have

\begin{align}
\norm{Q(t,s_1)x-Q(t,s_2)x}_X= \norm{\int_{s_1}^{s_2}U(t,r) Q(r) U(t,r)^\star  x dr}_X\leq M^2 K^2\norm{x}_X \abs{s_2-s_1} \nonumber
\end{align}
so that
\begin{align}
\norm{Q(t,s_1)-Q(t,s_2)}_{\op(X)}\leq N^2 K^2 \abs{s_2-s_1}, \nonumber
\end{align}
where $N$ and $K$ are the constants defined in Hypothesis \ref{1} and in \eqref{bshyp}.
Moreover
\begin{align}
\ix{Q(t,s_2) x}{x}&=\int_{s_2}^t \ix{U(t,r)\Qp{r}\Qp{r}U(t,r)^\star x}{x}\,dr=\nonumber \\
&=\int_{s_2}^t\norm{\Qp{r}U(t,r)^\star x}^2_X\,dr\leq\nonumber \\
&\leq\int_{s_1}^t\norm{\Qp{r}U(t,r)^\star x}^2_X\,dr=\ix{Q(t,s_1) x}{x}.\nonumber
\end{align}
\item The continuous embedding $\cm_{t,s_2}\subseteq\cm_{t,s_1}$ is an immediate consequence of proposition \ref{pseudo} and statement 1, observing that $\norm{\Qp{t,s} x}_X^2=\ix{Q(t,s) x}{x}$.
\item Let $x\in \Ker(Q(t,s_1))$. Since $$\displaystyle{\ix{Q(t,s_1) x}{x}=\int_{s_1}^t\norm{\Qp{r}U(t,r)^\star x}^2_X\,dr},$$ we obtain $\Qp{r}U(t,r)^\star x=0$ for almost every $r\in(s_1,t)$, so that $x\in \Ker(Q(t,s_2))$.

Let $x\in \Ker(Q(t,s_2))$. Then $\Qp{r}U(t,r)^\star x=0$ for almost every $r\in(s_2,t)$ and for $y\in X$ we have
\begin{align}\label{quasi ovunque}
0=\ix{\Qp{r}U(t,r)^\star x}{y}=\ix{x}{U(t,r)\Qp{r}y}, \ \ \mbox{a.e.}\ s_2<r<t.\nonumber
\end{align}
Letting $(r_n)_{n\in\Nat}\subset \{r\in[0,T]\ |\ s_2<r<t\ \mbox{and}\ \eqref{quasi ovunque}\ \mbox{holds}\}$ be such that $r_n\longrightarrow t$, we obtain 
\begin{align}
0=\ix{x}{\Qp{t}y}=\ix{\Qp{t}x}{y},\ \ \ \mbox{for all}\ \ y\in X.
\end{align}
Hence $\Qp{t}x=0$ and since $\Ker(Q(t))=\Ker(\Qp{t})$, the statement holds.
\end{enumerate}
\end{proof}

Here we extend some results of \cite[Sect.\,3]{MR1976297} to the non autonomous case.

\begin{prop}\label{cute} Let Hypothesis \ref{1} hold. We assume that for a fixed $s\in[0,T)$ there exists a normed space $E\subseteq X$ with continuous embedding and that $\uts(E)\subseteq \ac_t$ for every $t\in(s,T]$  
and that there exist $M>0$ and $\al\in\bigl[0,\frac{1}{2}\bigr)$ such that 
\begin{equation} \label{stimaesp2}
\norm{\uts}_{\op(E,\ac_t)}\leq \frac{M}{(t-s)^\al},\ \ t\in(s,T].
\end{equation}
Then for every $t\in(s,T]$, $\uts(E)\subseteq\cm_{t,s}$ and
{\small \begin{align}\label{van}
\norm{\uts}_{\op(E,\cm_{t,s})}=\norm{\Qm{t,s}\uts}_{\op(E,X)}\leq \frac{M}{(t-s)^{\frac{1}{2}+\al}}.
\end{align}}
\end{prop}
\begin{proof}
We consider the operator $L$ from remark \ref{3bis} with $\Qm{\si}$ instead of $B(\si)$. We know that $\re{L}=\re{\Qp{t,s}}$. On the other hand, for every $x\in E$ we have 
\begin{align}
(t-s)\uts x&=\int_s^t\uts x\,d\si=\int_s^t\utz U(\si,s)x\,d\si= \nonumber \\
&=\int_s^t\utz \Qp{\si}\Qm{\si}U(\si,s)x\,d\si.
\end{align}
where $y(\si):=\Qm{\si}U(\si,s)x$ belongs to $L^2\bigl((s,t);X\bigr)$. Hence $\uts x $ belongs to the range of $L$, and since $\al\in\bigl[0,\frac{1}{2}\bigr)$ by \eqref{Lpseudo} we get 
\begin{align}
\norm{\Qm{t,s}((t-s)\uts x)}_X&\leq\norm{y}_{L^2((s,t);X)}\leq \biggl(\int_s^t\norm{\Qm{\si} U(\si, s)x}^2_X\,d\si\biggr)^{\frac{1}{2}}\leq \nonumber \\
&\leq \biggl(\int_s^t\frac{M^2}{(\si-s)^{2\al}} \norm{x}_{E}^2\,d\si\biggr)^{\frac{1}{2}}\leq (t-s)^{\frac{1}{2}-\al} \frac{M}{(1-2\al)^{\frac{1}{2}}} \norm{x}_{E},\nonumber
\end{align}
which yields \eqref{van}.
\end{proof}

\begin{thm}\label{HS} Let Hypothesis \ref{1} hold. We assume that $\uts(\ac_s)\subseteq \ac_t$ for any $(s,t)\in\Delta$ 
and that there exist $M>0$ and $\al\in\bigl[0,\frac{1}{2}\bigr)$ such that 
\begin{equation} \label{stimaesp}
\norm{\uts}_{\op(\ac_s,\ac_t)}\leq \frac{M}{(t-s)^\al},\ \ (s,t)\in\Delta.
\end{equation}
Then the following statements hold.
\begin{enumerate}
\item For every $0\leq s < t\leq T$, $\uts(\ac_s)\subseteq\cm_{t,s}$ and
{\small \begin{align}
\norm{\uts}_{\op(\ac_s,\cm_{t,s})}=\norm{\Qm{t,s}\uts}_{\op(\ac_s,X)}\leq \frac{M}{(t-s)^{\frac{1}{2}-\al}}.
\end{align}}
\item For every $0\leq s < t\leq T$, $\uts^\star \bigl(\ker(Q(t))\bigr)\subseteq \Ker\bigl(Q(s)\bigr)$ and $\Ker\bigl(Q(t,s)\bigr)=\Ker\bigl(Q(t)\bigr)$.
\item For every $0\leq s_1<s_2 < t\leq T$, $\cm_{t,s_1}=\cm_{t,s_2}$ and their norms are equivalent.
\item For every $0\leq s < t\leq T$, $\cm_{t,s}$ is continuously embedded in $\ac_t$.
\end{enumerate}
\end{thm}
\begin{proof}
\leavevmode
\begin{enumerate}
\item Apply Proposition \ref{cute} with $E=\ac_s$.
\item We first remark some facts.
\begin{enumerate}
\item Since $\uts(\ac_s)\subseteq \ac_t$, by continuity we get $\uts(\overline{\ac_s})\subseteq \overline{\ac_t}$.
\item $\Ker(\Qp{t})=\Ker(Q(t))$, $\bigl(\Ker(Q(t))\bigr)^{\perp}=\overline{\ac_t}$ and consequently $\overline{\ac_t}=(I-P_t)(X)$, where we denote by $P_t$ the orthogonal projection on $\Ker(Q(t))$.
\item By statements (a) and (b) we have $\uts^\star (\ker(Q(t)))\subseteq \Ker(Q(s))$.
In fact
\begin{align}
\uts(\overline{\ac_s})\subseteq \overline{\ac_t}&\iff \uts(I-P_s)(X)\subseteq(I-P_t)(X)\nonumber \\
&\iff P_t \uts(I-P_s)= 0. \nonumber
\end{align}
Hence 
\begin{align}
(I-P_s)^\star \uts^\star P_t^\star=(I-P_s) \uts^\star P_t=0,
\end{align}
namely $\uts^\star (\ker(Q(t)))\subseteq \Ker(Q(s))$.
\end{enumerate}
 For every $x\in\Ker(Q(t))$, we have $U(t,r)^\star x\in\Ker\bigl(Q(r)\bigr)=\Ker\bigl(\Qp{r}\bigr)$ for every $r\in(s,t)$ and therefore $$\norm{\Qp{t,s} x}_X^2=\ix{Q(t,s) x}{x}=\int_{s}^t\norm{\Qp{r}U(t,r)^\star x}^2_X\,dr,$$ so that  $x\in\Ker\bigl(Q(t,s)\bigr)$. Conversely, the inclusion of $\Ker\bigl(Q(t,s)\bigr)$ in $\Ker(Q(t))$ follows from statement 3 of lemma \ref{cmQ}.
\item The continuous embedding $\cm_{t,s_2}\subseteq\cm_{t,s_1}$ is statement 2 of lemma \ref{cmQ}.
Concerning the reverse embedding, we first point out that the adjoint of the operator $\uts_{|_{\ac_s}}:\ac_s\longrightarrow \ac_t$ is the operator $(\uts_{|_{\ac_s}})^\star:\ac_t\longrightarrow \ac_s$ such that $$\iprod{x}{(\uts_{|_{\ac_s}})^\star y}_{\ac_s}=\iprod{\uts_{|_{\ac_s}}x}{y}_{\ac_t},\ \ \ \mbox{for all}\ x\in \ac_s,\ y\in \ac_t.$$ Now we claim that 
\begin{equation}\label{scambio}
(U(t,s)_{|_{\ac_s}})^\star Q(t) x=Q(s)\uts^\star  x,\ \ x\in X,\ (s,t)\in\Delta,
\end{equation}
where $\uts^\star $ denotes the adjoint operator of $U(t,s)\in\op(X)$. Indeed for all $h\in \ac_s$ we  have $(I-P_s)h=h$, $(I-P_t)U(t,s)h=U(t,s)h$ and
\begin{align}
\iprod{(U(t,s)_{|_{\ac_s}})^\star Q(t) x}{h}_{\ac_s}&=\iprod{Q(t) x}{U(t,s)_{|_{\ac_s}}h}_{\ac_t}=\ix{(I-P_t)x}{U(t,s)h}\nonumber \\
&=\ix{x}{(I-P_t)U(t,s)h}=\ix{\uts^\star x}{h}\nonumber\\
&=\ix{\uts^\star x}{(I-P_s)h}=\iprod{Q(s)\uts^\star x}{h}_{\ac_s}\nonumber\\
&=\iprod{Q(s)\uts^\star  x}{h}_{\ac_s}, \nonumber
\end{align}
where the last equality follows from statement 2.

Hence we get for $0\leq s_1<s_2<t\leq T$
{\small\begin{align}
\norm{\Qp{t,s_1} x}_X^2&=\int_{s_1}^t\norm{\Qp{r}U(t,r)^\star x}^2_X\,dr\nonumber \\
&=\int_{s_1}^{s_2}\norm{\Qp{r}U(t,r)^\star x}^2_X\,dr+\int_{s_2}^t\norm{\Qp{r}U(t,r)^\star x}^2_X\,dr \nonumber \\
&=\int_{s_2}^{2s_2-s_1}\norm{\Qp{\si_0+s_1-s_2}U(t,\si_0+s_1-s_2)^\star x}^2_X\,d\si_0+\norm{\Qp{t,s_2} x}_X^2, \nonumber
\end{align}}
where $\si_0=r-s_1+s_2$ in the first integral. 

If $t>2s_2-s_1$, by \eqref{scambio} we have
\begin{align}
&\int_{s_2}^{2s_2-s_1}\norm{\Qp{\si_0+s_1-s_2}U(t,\si_0+s_1-s_2)^\star x}^2_X\,d\si_0\nonumber \\
&\leq\int_{s_2}^{t}\norm{\Qp{\si_0+s_1-s_2}U(t,\si_0+s_1-s_2)^\star x}^2_X\,d\si_0\nonumber \\
&=\int_{s_2}^{t}\norm{Q(\si_0+s_1-s_2)U(\si_0,\si_0+s_1-s_2)^\star U(t,\si_0)^\star x}^2_{H_{\si_0+s_1-s_2}}\,d\si_0\nonumber \\
&=\int_{s_2}^{t}\norm{(U(\si_0,\si_0+s_1-s_2)_{|_{H_{\si_0+s_1-s_2}}})^\star Q(\si_0)U(t,\si_0)^\star x}^2_{H_{\si_0+s_1-s_2}}\,d\si_0\nonumber \\
&\leq M^2 \int_{s_2}^{t}\norm{ Q(\si_0)U(t,\si_0)^\star x}^2_{H_{\si_0+s_1-s_2}}\,d\si_0\nonumber \\
&= M^2 \int_{s_2}^{t}\norm{ \Qp{\si_0}U(t,\si_0)^\star x}^2_X\,d\si_0= M^2 \norm{\Qp{t,s_2}x}^2_X.\nonumber
\end{align}
Hence 
\begin{equation}\label{stimasi1}
\norm{\Qp{t,s_1} x}_X^2\leq (1+ M^2 )\norm{\Qp{t,s_2}x}^2_X. 
\end{equation}
If $s_2<t\leq 2s_2-s_1$, by \eqref{stimasi1} we have
{\small \begin{align}
&\int_{s_2}^{2s_2-s_1}\norm{\Qp{\si_0+s_1-s_2}U(t,\si_0+s_1-s_2)^\star x}^2_X\,d\si_0\nonumber \\
&=\int_{s_2}^t+\int_t^{2s_2-s_1}\norm{\Qp{\si_0+s_1-s_2}U(t,\si_0+s_1-s_2)^\star x}^2_X\,d\si_0\nonumber\\
&\leq M^2 \norm{\Qp{t,s_2}x}^2_X
+\int_{s_2}^{3s_2-s_1-t}\norm{\Qp{\si_1+t+s_1-2s_2}U(t,\si_1+t+s_1-2s_2)^\star x}^2_X\,d\si_1, \nonumber
\end{align}}
where $\si_1=\si_0-t+s_2$.

If $t>3s_2-s_1-t\iff t>\frac{3s_2-s_1}{2}$, hence for any $t\in \Bigl(\frac{3s_2-s_1}{2},2s_2-s_1\Bigr]$ we have
\begin{align}
&\int_{s_2}^{3s_2-s_1-t}\norm{\Qp{\si_1+t+s_1-2s_2}U(t,\si_1+t+s_1-2s_2)^\star x}^2_X\,d\si_0\nonumber \\
&\leq\int_{s_2}^{t}\norm{\Qp{\si_1+t+s_1-2s_2}U(t,\si_1+t+s_1-2s_2)^\star x}^2_X\,d\si_1 \nonumber
\end{align}
and by \eqref{scambio}
\begin{align}
&\int_{s_2}^{t}\norm{\Qp{\si_1+t+s_1-2s_2}U(t,\si_1+t+s_1-2s_2)^\star x}^2_X\,d\si_1\nonumber \\
&=\int_{s_2}^{t}\norm{Q(\si_1+t+s_1-2s_2)U(\si_1,\si_1+t+s_1-2s_2)^\star U(t,\si_1)^\star x}^2_{H_{\si_1+t+s_1-2s_2}}\,d\si_1\nonumber \\
&=\int_{s_2}^{t}\norm{(U(\si_1,\si_1+t+s_1-2s_2)_{H_{\si_1+t+s_1-2s_2}})^\star Q(\si_1)U^\star(t,\si_1)x}^2_{H_{\si_1+t+s_1-2s_2}}\,d\si_1\nonumber \\
&\leq M^2 \norm{\Qp{t,s_2}}^2_X. \nonumber
\end{align}
Adding up,
\begin{equation} \label{primostepA}
\norm{\Qp{t,s_1} x}_X^2\leq (1+ 2M^2)\norm{\Qp{t,s_2}x}^2_X. 
\end{equation}

Instead, if $\displaystyle{s_2<t\leq \frac{3s_2-s_1}{2}}$
{\small \begin{align}\label{primostepB}
&\int_{s_2}^{3s_2-s_1-t}\norm{\Qup(\si_1+t+s_1-2s_2)U^\star(t,\si_1+t+s_1-2s_2)x}^2_X\,d\si_1\nonumber \\
&=\int_{s_2}^t+\int_t^{3s_2-s_1-t}\norm{\Qup(\si_1+t+s_1-2s_2)U^\star(t,\si_1+t+s_1-2s_2)x}^2_X\,d\si_1\nonumber\\
&\leq \int_{s_2}^{4s_2-s_1-2t}\norm{\Qup(\si_2+2t+s_1-3s_2)U^\star(t,\si_2+2t+s_1-3s_2)x}^2_X\,d\si_2+M^2 \norm{\Qp{t,s_2}x}^2_X, 
\end{align}}
where $\si_2=\si_1-t+s_2$.

Let $k\in\Nat$. Now we prove the following. 
\begin{itemize}
\item If $t\in\Bigl(\frac{(k+2)s_2-s_1}{k+1},\frac{(k+1)s_2-s_1}{k}\Bigr]$, then
\begin{align}\label{AI}
\norm{\Qp{t,s_1} x}_X^2\leq \bigl[1+(k+1) M^2 \bigr]\norm{\Qp{t,s_2}x}^2_X. 
\end{align}
\item If $t\in\Bigl(s_2,\frac{(k+2)s_2-s_1}{k+1}\Bigl]$, then
\begin{align} \label{BI}
&\int_{s_2}^{s_k}\norm{\Qup\bigl(t_k\bigr)U^{\star}\bigl(t,t_k\bigr)}^2_X\,d\si_k\leq \int_{s_2}^{s_{k+1}}\norm{\Qup\bigl(t_{k+1}\bigr)U^{\star}\bigl(t,t_{k+1}\bigr)}^2_X\,d\si_{k+1}\nonumber\\
&+M^2\norm{\Qp{t,s_2}}^2_X, 
\end{align}
where 
\begin{align}
&s_k=(k+2)s_2-s_1-kt, \nonumber \\
&t_k=\si_k+kt+s_1-(k+1)s_2, \nonumber \\
&\si_{k+1}=\si_k-t+s_2. \nonumber
\end{align}
\end{itemize}
We proceed by recurrence over $k$. 

For $k=1$ \eqref{AI} and \eqref{BI} have been already proved in \eqref{primostepA} and \eqref{primostepB}. 

We assume now that \eqref{AI} and \eqref{BI} hold for a given $k\in\Nat$ and we prove that they hold for $k+1$.

Let $t\in\Bigl(s_2,\frac{(k+2)s_2-s_1}{k+1}\Bigl]$.

 If $t>s_{k+1}\iff t>\frac{(k+3)s_2-s_1}{k+2}$,  for any $t\in \Bigl(\frac{(k+3)s_2-s_1}{k+2},\frac{(k+2)s_2-s_1}{k+1}\Bigr]$ we have
\begin{align}
\int_{s_2}^{s_{k+1}}\norm{\Qup\bigl(t_{k+1}\bigr)U^{\star}\bigl(t,t_{k+1}\bigr)}^2_X\,d\si_{k+1}\leq\int_{s_2}^{t}\norm{\Qup\bigl(t_{k+1}\bigr)U^{\star}\bigl(t,t_{k+1}\bigr)}^2_X\,d\si_{k+1}\nonumber
\end{align}
and by \eqref{scambio}
\begin{align}
&\int_{s_2}^{t}\norm{\Qp{t_{k+1}}\bigr)U(t,t_{k+1})^{\star}}^2_X\,d\si_{k+1}=\int_{s_2}^{t}\norm{Q(t_{k+1})U(t,t_{k+1})^\star x}^2_{H_{t_{k+1}}}\,d\si_{k+1}=\nonumber \\
&=\int_{s_2}^{t}\norm{(U(\si_{k+1},t_{k+1})_{|_{H_{t_{k+1}}}})^\star Q(\si_{k+1})U(t,\si_{k+1})^\star x}^2_{H_{t_{k+1}}}\,d\si_{k+1}=\nonumber \\
&\leq M^2 \norm{\Qp{t,s_2}}^2_X. \nonumber 
\end{align}
Moreover
\begin{align}
\norm{\Qp{t,s_1} x}_X^2&\leq \bigl[1+ (k+2)M^2 \bigr]\norm{\Qp{t,s_2}x}^2_X,\nonumber
\end{align}
and \eqref{AI} holds for $k+1$.

Now if $t\in\Bigl(s_2,\frac{(k+3)s_2-s_1}{k+2}\Bigl]$, then
\begin{align} 
&\int_{s_2}^{s_{k+1}}\norm{\Qp{t_{k+1}}U(t,t_{k+1})^{\star}}^2_X\,d\si_{k+1}=\int_{s_2}^{t}+\int_{t}^{s_{k+1}}\norm{\Qp{t_{k+1}}U(t,t_{k+1})^{\star}}^2_X\,d\si_{k+1}\nonumber \\
&\leq \int_{s_2}^{s_{k+2}}\norm{\Qp{t_{k+2}}\bigr)U(t,t_{k+2})^{\star}}^2_X\,d\si_{k+2}+M^2 \norm{\Qp{t,s_2}}^2_X. \nonumber
\end{align}
and \eqref{BI} holds for $k+1$.

Therefore if $k\in\Nat$ and $t\in\Bigl(\frac{(k+2)s_2-s_1}{k+1},\frac{(k+1)s_2-s_1}{k}\Bigr]$,  
\begin{align}
\norm{\Qp{t,s_1} x}_X^2&\leq \bigl[1+(k+1) M^2 \bigr]\norm{\Qp{t,s_2}x}^2_X. \nonumber
\end{align}
Since $$\bigcup_{k\in\Nat}\Biggl(\frac{(k+2)s_2-s_1}{k+1},\frac{(k+1)s_2-s_1}{k}\Biggr]=(s_2,2s_2-s_1]$$ and \eqref{stimasi1} holds for $t\in(2s_2-s_1,+\infty)$, statement 3 follows.
\item For $(s,t)\in\Delta$ and $x\in X$ we use again \eqref{scambio}, to get
\begin{align}
\norm{\Qp x}_X&=\int_s^t\norm{\Qp{r}U(t,r)^\star x}^2_X\,dr=\int_s^t\norm{Q(r)U(t,r)^\star x}^2_{H_r}\,dr\nonumber\\
&=\int_s^t\norm{(U(t,r)_{|_{H_r}})^\star Q(t)x}^2_{H_r}\,dr\leq M^2 (t-s) \norm{\Qp{t}x}^2_X\nonumber.
\end{align}
By Proposition \ref{pseudo} we the statement follows.
\end{enumerate}
\end{proof}


\section{Schauder type theorems} 
In this section we assume that Hypotheses \ref{1} and \ref{3} hold and we prove maximal H\"older regularity for the mild solution of problem \eqref{bwp}, that is given \eqref{solmild}. We argue as in the papers \cite{MR4311102} and \cite{cerlun}. 
We set 
\begin{align}
u_0(s,x)&=\pst\ph(x),\ \ (s,t)\in\Delta,\ x\in X \\
u_1(s,x)&=-\int_s^t\bigl(P_{s,\si}\psi(\si,\cdot)\bigr)(x)\,d\si,\ \ (s,t)\in\Delta,\ x\in X.
\end{align}

\begin{prop}\label{passozero}
For every $t\in[0,T]$ and $\psi\in C_b\bigl([0,t]\times X\bigr)$, $u_1$ is continuous and
\begin{equation}\label{stimau1}
\norm{u_1}_\infty\leq t\norm{\psi}_\infty.
\end{equation}
Moreover the following statements hold.
\begin{enumerate}
\item[(1)] Let $\te<1$. For every $n\in\Nat$ such that $\displaystyle{n<\frac{1}{\te}}$, the function $u_1$ belongs to $C^{0,n}_E\bigl([0,t]\times X\bigr)$ and there exists $C>0$, independent of $\psi$ and $t$, such that
\begin{equation}\label{stimau1inf}
\norm{u_1}_{C^{0,n}_E([0,t]\times X)}\leq C \norm{\psi}_\infty.
\end{equation}
\item[(2)] Let $\al\in(0,1)$ be such that $\displaystyle{\al+\frac{1}{\te}>1}$. For every $\psi\in C^\al_E(X)$ and for every $n\in\Nat$ such that $\displaystyle{n<\al+\frac{1}{\te}}$, the function $u_1$ belongs to $C^{0,n}_E\bigl([0,t]\times X\bigr)$ and there exists $C>0$, independent of $\psi$ and $t$, such that
\begin{equation}\label{stimau1al}
\norm{u_1}_{C^{0,n}_E([0,t]\times X)}\leq C\norm{\psi}_{C^{0,\al}_E([0,t]\times X)}.
\end{equation}
\end{enumerate}
\end{prop}

\begin{proof}
Estimate \eqref{stimau1} is obvious and the proof that u is continuous is the same as \cite[Prop.\, 3.2]{cerlun}.
%
%

Concerning statements (1) and (2), due to Corollary \ref{blowup} we can differentiate in the integral that defines $u_1$ to obtain  
\begin{equation}
\frac{\de u_1}{\de h_1,...,\de h_k}(s,x)=\int_s^t D^k P_{s,\si}\psi(\si,\cdot)(x)(h_1,...,h_k)\, d\si,\ \ s\in[0,t),\ x\in X,
\end{equation}
for every $h_1,...,h_k\in E$ and $k\in \{1,...,n\}$. We prove now that the mapping 
\begin{equation}\label{derivateu1}
(s,x)\in[0,t)\times X\longmapsto D^k_Eu_1(s,\cdot)(x)(h_1,...,h_k)\in\R,
\end{equation}
 is continuous for every $h_1,...,h_k\in E$ and $k\in \{1,...,n\}$.  
 
 For $0\leq s_0\leq s<t\leq T$ and $x,x_0\in X$, we write
\begin{align} \label{tanteh}
\abs{\frac{\de u_1}{\de h_1,...,\de h_k}(s,x)-\frac{\de u_1}{\de h_1,...,\de h_k}(s_0,x_0)}=\int_{s_0}^s \abs{D^k_E P_{s_0,\si}\psi(\si,\cdot)(x_0)(h_1,...,h_k)}\, d\si \\
+\int_{s_0}^t\ca{[s,t]}(\si) \abs{D^k_E P_{s,\si}\psi(\si,\cdot)(x)(h_1,...,h_k)-D^k _E P_{s_0,\si}\psi(\si,\cdot)(x)(h_1,...,h_k)}\, d\si. \nonumber
\end{align}
Since for every $\si\geq s_0$ the mapping $$(s,x)\in[0,T]\times X\longmapsto \ca{[s,t]}(\si)D^k_E P_{s,\si}\psi(\si,\cdot)(x)(h_1,...,h_k)\in\R,$$ is continuous by Lemma \ref{regck2} and by Lemma \ref{regck} and $$\abs{D^k_E P_{s,\si}\psi(\si,\cdot)(x)(h_1,...,h_k)-D^k_EP_{s_0,\si}\psi(\si,\cdot)(x_0)(h_1,...,h_k)}\leq2 M^k\sup_{x\in X}\norm{D^k_E \psi}_{\op^k{\ets}} \prod_{i=1}^k \norm{h_i}_E,$$
by the Dominated Convergence Theorem the first integral in \ref{tanteh} vanishes as $s\rightarrow s_0^+$ and $x\rightarrow x_0$. 
Moreover, since
$$\int_{s_0}^s\abs{D_E^kP_{s_0,\si}\psi(\si,\cdot)(x_0)(h_1,...,h_k)}\,d\si\leq M^k\norm{D_E^k\psi}_\infty\prod_{i=1}^k\norm{h_i}_E (s-s_0),$$even the second integral in \ref{tanteh} vanishes as $s\rightarrow s_0^+$ and $x\rightarrow x_0$.

 If $s<s_0$ we split 
 \begin{align}
 &\abs{\frac{\de u_1}{\de h_1,...,\de h_k}(s,x)-\frac{\de u_1}{\de h_1,...,\de h_k}(s_0,x_0)}\leq \int^{s_0}_s\abs{D_E^kP_{s_0,\si}\psi(\si,\cdot)(x_0)(h_1,...,h_k)}\,d\si\nonumber\\
 &+\int_{s_0}^t \abs{D_E^kP_{s,\si}\psi(\si,\cdot)(x)(h_1,...,h_k)-D_E^kP_{s_0,\si}\psi(\si,\cdot)(x_0)(h_1,...,h_k)}\,d\si\nonumber
 \end{align}
  and we proceed as before to show that $\displaystyle{\frac{\de u_1}{\de h_1,...,\de h_k}}$ is continuous.

\end{proof}

\begin{thm}\label{schau1}
Assume that Hypotheses \ref{1} and \ref{3} hold.  Let $\ph\in C_b(X)$, $\psi\in C_b\bigl([0,t]\times X\bigr)$ and let $u$ be defined by \eqref{solmild}. 
\begin{enumerate}
\item[(1)] If  $\displaystyle{\frac{1}{\te}\notin\Nat}$, $\ph\in C^{\frac{1}{\te}}_E(X)$ and $\psi\in C_b\bigl([0,t]\times X\bigr)$, then $u\in C_E^{0,\frac{1}{\te}}\bigl([0,t]\times X\bigr)$. Moreover there exists $C=C(T)>0$, independent of $\ph$ and $\psi$, such that 
\begin{equation}
\norm{u}_{C_E^{0,\frac{1}{\te}}([0,t]\times X)}\leq C\bigl(\norm{\ph}_{C^{\frac{1}{\te}}_E(X)}+\norm{\psi}_{\infty}\bigr).
\end{equation}
\item[(2)] If  $\al\in(0,1)$, $\displaystyle{\al+\frac{1}{\te}\notin\Nat}$, $\ph\in C^{\al+\frac{1}{\te}}_E(X)$ and $\psi\in C^{0,\al}_E\bigl([0,t]\times X\bigr)$, then $u\in C_E^{0,\al+\frac{1}{\te}}\bigl([0,t]\times X\bigr)$. Moreover there exists $C=C(T,\al)>0$, independent of $\ph$ and $\psi$, such that 
\begin{equation}
\norm{u}_{C_E^{0,\al+\frac{1}{\te}}([0,t]\times X)}\leq C\bigl(\norm{\ph}_{C^{\al+\frac{1}{\te}}_E(X)}+\norm{\psi}_{C_E^{0,\al}([0,t]\times X)}\bigr).
\end{equation}
\end{enumerate}
\end{thm}

\begin{proof}
We can assume that $\ph\equiv 0$ since, by Corollary \ref{u0}, we know that for every  $\g>0$ $u_0$ belongs to $C^\g_E([0,t]\times X)$, if $\ph\in C^\g_E(X)$; moreover there exists $C=C(\g, T)>0$ such that
\begin{equation}
\norm{u_0}_{C^{0,\g}_E([0,t]\times X)}\leq C\norm{\ph}_{C^\g_E(X)}.
\end{equation}
Since this proof is similar to \cite[Thm.\,3.3]{cerlun}, we only prove statement (2).
Setting $\displaystyle{n:=\biggl[\al+\frac{1}{\te}\biggr]}$, we have  $$(n-\al)\te\in(0,1),\ \ (n+1-\al)\te>1.$$ We already know that $u_1\in C^{0,n}_E(X)$ by Proposition \ref{passozero} (2) and we have to show that $u_1(s,\cdot)\in C^{\al+\frac{1}{\te}}_E(X)$ for every $s\in[0,t]$.

If $n>0$, we have to prove that $D^n_Eu_1(s,\cdot)$ is $\displaystyle{\biggl(\al+\frac{1}{\te}-n\biggr)}$-H\"older continuous with values in $\op^n(E)$, with H\"older constant independent of $s$. 

Let $h,h_1,...h_n\in E$. We split every partial derivative $D^n_Eu_1(s,y)(h_1,...h_n)$ into $a_h(s,y)+b_h(s,y)$ where
\begin{align}\label{spezzo1}
a_h(s,y)&:=-\int_s^{(s+\norm{h}_E^{\frac{1}{\te}})\wedge t} D^n_E P_{s,\si}\psi(\si,\cdot)(y)(h_1,...,h_n)\,d\si,\ \ s\in[0,t],\ y\in X,\\
b_h(s,y)&:=-\int^t_{(s+\norm{h}_E^{\frac{1}{\te}})\wedge t} D^n_E P_{s,\si}\psi(\si,\cdot)(y)(h_1,...,h_n)\, d\si,\ \ s\in[0,t],\ y\in X. \label{spezzo2}
\end{align}
Due to \eqref{kal}, we obtain 
\begin{align}
&\abs{a_h(s,x+h)-a_h(s,x)}\leq\abs{a_h(s,x+h)}+\abs{a_h(s,x)}\nonumber \\
&\leq 2 K_{n,\al}\norm{\psi}_{C^{\al}_E([0,t]\times X)}\prod_{j=1}^n\norm{h_j}_E\int_s^{(s+\norm{h}_E^{\frac{1}{\te}})\wedge t} (\si-s)^{-(n-\al)\si}\,d\si.\nonumber \\
&\leq \frac{2K_{n,\al}}{1-(n-\al)\te}\norm{h}_E^{\frac{1-(n-\al)\te}{\te}}\norm{\psi}_{C^{\al}_E([0,t]\times X)}\prod_{j=1}^n\norm{h_j}_E.\nonumber
\end{align}
We observe that if $\norm{h}_E^{\frac{1}{\te}}\geq t-s$, $b_h(s,\cdot)$ vanishes. Hence, we estimate now $\abs{b_h(s,x+h)-b_h(s,x)}$ when $\norm{h}_E^{\frac{1}{\te}}< t-s$. Again, by \eqref{kal} we get
\begin{align}
&\norm{D^n_E P_{s,\si}\psi(\si,\cdot)(x+h)-D^n_E P_{s,\si}\psi(\si,\cdot)(x)}_{\op^n(E)}\leq\sup_{y\in X}\norm{D^{n+1}_E P_{s,\si}\psi(\si,\cdot)(y)}_{\op^n(E)}\norm{h}_E\nonumber\\
&\leq\frac{K_{n+1,\al}}{(\si-s)^{(n+1-\al)\te}}\norm{\psi}_{C^{\al}_E([0,t]\times X)}\norm{h}_E,\ \si\in(s,t)
\end{align}
and
{\small \begin{align}
&\abs{b_h(s,x+h)-b_h(s,x)}\leq  K_{n+1,\al}\norm{\psi}_{C^{\al}_E([0,t]\times X)}\norm{h}_E\prod_{j=1}^n\norm{h_j}_E\int^t_{(s+\norm{h}_E^{\frac{1}{\te}})\wedge t} (\si-s)^{-(n+1-\al)\te}\,d\si.\nonumber \\
&\leq \frac{K_{n+1,\al}}{(n+1-\al)\te-1}\norm{h}_E^{\frac{1}{\te}+\al-n}\norm{\psi}_{C^{\al}_E([0,t]\times X)}\prod_{j=1}^n\norm{h_j}_E.\nonumber
\end{align}}
Summing up we get $$\abs{\bigl(D^n_Eu_1(s,x+h)-D^n_Eu_1(s,x)\bigr)(h_1,...,h_n)}\leq C\norm{h}^{\frac{1}{\te}+\al-n}_E\norm{\psi}_{C^{\al}_E([0,t]\times X)}\prod_{j=1}^n\norm{h_j}_E$$ with $$C=\frac{2K_{n,\al}}{1-(n-\al)\te}+\frac{K_{n+1,\al}}{(n+1-\al)\te-1}.$$
Therefore, $$[D^n_Eu_1(s,\cdot)]_{C^{\frac{1}{\te}+\al-n}_E(X;\op^n(E))}\leq C\norm{\psi}_{C^{\al}_E([0,t]\times X)}.$$
The case $n=0$ is analogous to the previous one where instead of \eqref{kal} we use \eqref{stimau1}.

\end{proof}

\begin{thm}\label{schau2}
Assume that Hypotheses \ref{1} and \ref{3} hold. Let $\ph\in C_b(X)$, $\psi\in C_b\bigl([0,t]\times X\bigr)$ and let $u$ be defined by \eqref{solmild}.
\begin{enumerate}
\item[(1)] If  $\displaystyle{\frac{1}{\te}=k\in\Nat}$ and if $\ph\in Z^k_E(X)$, then $u\in Z_E^{0,k}\bigl([0,t]\times X\bigr)$. Moreover there exists $C=C(T)>0$, independent of $\ph$ and $\psi$, such that 
\begin{equation}
\norm{u}_{Z_E^{0,k}([0,t]\times X)}\leq C\bigl(\norm{\ph}_{Z^k_E(X)}+\norm{\psi}_{\infty}\bigr).
\end{equation}
\item[(2)] If  $\al\in(0,1)$ and $\displaystyle{\al+\frac{1}{\te}=k\in\Nat}$ and if $\ph\in Z^k_E(X)$ and $\psi\in C^{0,\al}_E\bigl([0,t]\times X\bigr)$, then $u\in Z_E^{0,k}\bigl([0,t]\times X\bigr)$. Moreover there exists $C=C(T,\al)>0$, independent of $\ph$ and $\psi$, such that 
\begin{equation}
\norm{u}_{Z_E^{0,k}([0,t]\times X)}\leq C\bigl(\norm{\ph}_{Z^{k}_E(X)}+\norm{\psi}_{C_E^{0,\al}([0,t]\times X)}\bigr).
\end{equation}
\end{enumerate}
\end{thm}

\begin{proof}
We can assume that $\ph\equiv 0$ since, by Corollary \ref{u0}, we know that for every $k\in\Nat$ $u_0$ belongs to $Z^k_E([0,t]\times X)$, if $\ph\in Z^k_E(X)$; moreover there exists $C=C(k, T)>0$ such that
\begin{equation}
\norm{u_0}_{Z^{0,k}_E([0,t]\times X)}\leq C\norm{\ph}_{Z^k_E(X)}.
\end{equation}

Since this proof is similar to \cite[Thm.\,3.4.]{cerlun}, we only prove statement (1). In this case  we have $\psi\in C_b\bigl([s,t]\times X\bigr)$ and $\displaystyle{k=\frac{1}{\te}}$.  If $k\geq 2$, we know from Proposition \ref{passozero} that $u_1\in C^{0,k-1}_E\bigl([0,t]\times X\bigr)$. So we have to show that $[D^{k-1}_Eu_1(s,\cdot)]_{Z^1_E(X;\op^{k-1}(E))}$ is bounded by a constant independent of $s$. 

 Let $h,h_1,...h_n\in E$. We split every partial derivative $D^{k-1}_Eu_1(s,y)(h_1,...h_n)$ into the sum $a_h(s,y)+b_h(s,y)$ defined in \eqref{spezzo1} and \eqref{spezzo2}, as we did in the Theorem \ref{schau1}.
Due to \eqref{kn}, we obtain 
\begin{align}
&\abs{a_h(s,x+2h)-2a_h(s,x+h)+a_h(s,x)}\leq\abs{a_h(s,x+2h)}+2\abs{a_h(s,x+h)}+\abs{a_h(s,x)}\nonumber \\
&\leq 4 K_{k-1}\norm{\psi}_\infty\prod_{j=1}^{k-1}\norm{h_j}_E\int_s^{(s+\norm{h}_E^{\frac{1}{\te}})\wedge t} (\si-s)^{-(k-1)\si}\,d\si.\nonumber \\
&\leq \frac{4K_{k-1}}{1-(k-1)\te}\norm{h}_E^{\frac{1-(k-1)\te}{\te}}\norm{\psi}_\infty\prod_{j=1}^n\norm{h_j}_E.\nonumber\\
&=4 k K_{k-1}\norm{h}_E\norm{\psi}_\infty\prod_{j=1}^n\norm{h_j}_E. \nonumber
\end{align}
We observe that if $\norm{h}_E^{\frac{1}{\te}}\geq t-s$, $b_h(s,\cdot)$ vanishes. Hence we estimate $$\abs{b_h(s,x+2h)-2b_h(s,x+h)+b_h(s,x)}$$ when $\norm{h}_E^{\frac{1}{\te}}< t-s$. Again, by \eqref{kn} we get
\begin{align}
&\abs{b_h(s,x+2h)-2b_h(s,x+h)+b_h(s,x)}\nonumber \\
&\leq\int_{s+\norm{h}_E^{\frac{1}{\te}}}^t\Bigl|\bigl(D^{k-1}_EP_{s,\si}\psi(\si,\cdot)(x+2h)-2D^{k-1}_EP_{s,\si}\psi(\si,\cdot)(x+h)\nonumber\\
&+D^{k-1}_EP_{s,\si}\psi(\si,\cdot)(x)\bigr)(h_1,...,h_{k-1})\Bigr|\, d\si\nonumber \\
&\leq \norm{h}_E^2 \prod_{j=1}^{k-1}\norm{h_j}_E\int_{s+\norm{h}_E^{\frac{1}{\te}}}^t\sup_{y\in X}\norm{D^{k+1}_EP_{s,\si}\psi(\si,\cdot)(y)}_{\op^{k+1}}\,d\si\nonumber\\
&\leq K_{k+1}\norm{\psi}_\infty\norm{h}_E^2 \prod_{j=1}^{k-1}\norm{h_j}_E\int_{s+\norm{h}_E^{\frac{1}{\te}}}^t (\si-s)^{-(k+1)\te}\, d\si\leq  k K_{k+1}\norm{\psi}_\infty\norm{h}_E \prod_{j=1}^{k-1}\norm{h_j}_E.\nonumber
\end{align}
Summing up we get $$[D^{k-1}_Eu_1(s,\cdot)]_{Z^1_E(X;\op^{k-1}(E))}\leq k(4K_{k-1}+K_{k+1})\norm{\psi}_\infty.$$
The case $k=\te=1$ is analogous to the previous one where instead of \eqref{kn} we use \eqref{stimau1}.

\end{proof}

\section[]{Examples}

\subsection{Example 1}
 Let $T>0$ and let $A(t)$, $B(t)$ be self-adjoint operators in diagonal form with respect to the same Hilbert basis $\{e_k:\ k\in\Nat\}$, namely $$A(t)e_k=a_k(t)e_k,\ \ B(t)e_k=b_k(t)e_k\ \ 0\leq t\leq T,\ \ k\in\Nat,$$ with continuous coefficients $a_k$, $b_k$. We set $$\mu_k=\min_{t\in [0,T]}a_k(t),\ \ \la_k=\max_{t\in [0,T]}a_k(t)$$ and we assume that there exists $\la_0>0$ such that $$\la_k\leq\la_0,\ \ \forall\ k\in\Nat.$$
In this setting  the operator $\uts$ is defined by $$U(t,s)e_k=\exp\biggl(\int_s^ta_k(\tau)\,d\tau\biggr)e_k,\ \ (s,t)\in\Delta,\ \ k\in\Nat$$is the strongly continuous evolution operator associated to the family $\{A(t)\}_{t\in[0,T]}$.
Moreover we assume that there exists $K>0$ such that $$\abs{b_k(t)}\leq K,\ \ t\in[0,T],\ k\in\Nat.$$
Hence $B(t)\in\op(X)$ for all $t\in[0,T]$, the function $B:[0,T]\longmapsto\op(X)$ is continuous and $$\sup_{t\in[0,T]}\norm{B(t)}_{\op(X)}\leq K.$$
The operators $Q(t,s)$ are given by
\begin{equation}
Q(t,s)e_k=\int_s^t\exp\biggl(2\int_\si^ta_k(\tau)\,d\tau\biggr)(b_k(\si))^2\,d\si\,e_k=:q_k(t,s)e_k,\ \ (s,t)\in\Delta,\ k\in\Nat. \nonumber
\end{equation}
Hypothesis \ref{1} is fullfilled if 
\begin{equation}\label{qk1}
\sum_{k=0}^\infty q_k(t,s)<+\infty,\ \ (s,t)\in\Delta.
\end{equation}
We give now a sufficient condition for \eqref{qk1} to hold.
We assume that $\la_k$ is eventually nonzero (say for $k\geq k_0$). Given $(s,t)\in\Delta$, we have
\begin{align}
&\abs{\int_s^t\exp\biggl(2\int_\si^ta_k(\tau)\,d\tau\biggr)(b_k(\si))^2\,d\si}\leq \norm{b_k}^2_{\infty}\abs{\int_s^t\exp\bigl(2\la_k(t-\si)\bigr)\,d\si}\nonumber \\
&=\frac{\norm{b_k}^2_{\infty}}{2\abs{\la_k}}\abs{1-\exp(2\la_k(t-s))}\leq\frac{\norm{b_k}^2_{\infty}}{2\abs{\la_k}}\bigl(1+\exp(2\la_0T)\bigr).
\end{align}
Hence \eqref{qk1} holds if we require 
\begin{equation}
\sum_{k=k_0}^\infty\frac{\norm{b_k}^2_{\infty}}{\abs{\la_k}}<+\infty.
\end{equation}
In \cite{cerlun} sufficient conditions for $\uts(X)\subseteq\cm_{t,s}$ were given. We look now for a condition such that there exists $(s,t)\in\Delta$ such that
\begin{equation} \label{notsf1}
U(t,s)(X)\nsubseteq \cm_{t,s}.
\end{equation}
Since $\Qp{t,s}e_k=(q_k(t,s))^{\frac{1}{2}}e_k$ for every $k\in\Nat$, \eqref{notsf1} holds if either some of the $q_k(t,s)$ vanishes (so that the measure $\Ng_{0,Q(t,s)}$ is degenerate) or if none of the $q_k(t,s)$ vanishes, but
\begin{equation} \label{bum}
\sup_{\substack{k\in\Nat}}\frac{\displaystyle{\exp\biggl(\int_s^t2a_k(\tau)\,d\tau\biggr)}}{q_k(t,s)}=+\infty.
\end{equation}
We obtain a sufficient condition for \eqref{bum} if we assume that $\la_k$ is eventually negative (say for $k\geq k_1\geq k_0$). We have  
\begin{align}
&\frac{\displaystyle{\exp\biggl(\int_s^t2a_k(\tau)\,d\tau\biggr)}}{\displaystyle{\int_s^t\exp\biggl(2\int_\si^ta_k(\tau)\,d\tau\biggr)(b_k(\si))^2\,d\si}}\geq\frac{\exp\bigl(2\mu_k(t-s)\bigr)}{\displaystyle{\norm{b_k}^2_{\infty}\int_s^t\exp\bigl(2\la_k(t-\si)\bigr)\,d\si}}\nonumber \\
&=\frac{\exp\bigl(2\mu_k(t-s)\bigr)}{\displaystyle{\norm{b_k}^2_{\infty}\frac{1}{-2\la_k}\bigl(1-\exp(2\la_k(t-s))\bigr)}}=\nonumber\\
&=\frac{2\abs{\la_k}}{\displaystyle{\norm{b_k}^2_{\infty}\bigl[\exp\bigl(-2\mu_k(t-s)\bigr)-\exp\bigl(2(\la_k-\mu_k)(t-s)\bigr)\bigr]}}. \nonumber
\end{align}
So \eqref{bum} is fullfilled if 
\begin{equation}
\sup_{\substack{k\geq k_1}}\frac{2\abs{\la_k}}{\displaystyle{\norm{b_k}^2_{\infty}\bigl(\exp\bigl(-2\mu_k(t-s)\bigr)-\exp(-2(\mu_k-\la_k)(t-s))\bigr)}}=+\infty,\ \ \mbox{for some}\ (s,t)\in\Delta.
\end{equation}
Now we want to find some conditions such that the hypotheses of Theorem \ref{HS} are satisfied. Let $(s,t)\in\Delta$, we assume that $b_k(t)$ is nonzero for all $k\in\Nat$. Hence we have $U(t,s)H_s\subseteq H_t$ if
\begin{equation}\label{primaip}
\sup_{k\in\Nat}\frac{b_k^2(s)}{b_k^2(t)}\exp\biggl(\int_s^t2a_k(\tau)\,d\tau\biggr)<+\infty.
\end{equation}  
Since 
\begin{equation}
\frac{b_k^2(s)}{b_k^2(t)}\exp\biggl(\int_s^t2a_k(\tau)\biggr)\leq \frac{b_k^2(s)}{b_k^2(t)} \exp(2\la_k T),\ \ k\in\Nat\nonumber
\end{equation}
a sufficient condition for \eqref{primaip} to hold is 
\begin{equation}
\sup_{k\in\Nat}\frac{b_k^2(s)}{b_k^2(t)}\exp(2\la_k T)<+\infty.
\end{equation}
Now we investigate when \eqref{stimaesp} holds.
We observe first that for $y=\Qp{s}x$ with $x\in H_s$, we have 
\begin{align}
&\norm{U(t,s)y}^2_{H_t}=\norm{\Qm{t}U(t,s)\Qp{s}x}_X^2=\sum_{\substack{k\in\Nat}}\biggl(\frac{\abs{b_k(s)}}{\abs{b_k(t)}}\exp\biggl(\int_s^ta_k(\tau)\biggr)\iprod{x}{e_k}\biggr)^2\nonumber\\
&=\sum_{\substack{k\in\Nat}}\biggl(\frac{\abs{b_k(s)}}{\abs{b_k(t)}}\exp\biggl(\int_s^ta_k(\tau)\biggr)\frac{\iprod{y}{e_k}}{\abs{b_k(s)}}\biggr)^2=\sum_{\substack{k\in\Nat}}\biggl(\frac{1}{\abs{b_k(t)}}\exp\biggl(\int_s^ta_k(\tau)\iprod{y}{e_k}_{H_s}\biggr)\biggr)^2. \nonumber
\end{align}

We assume that there exist $L\geq 0$ such that $\abs{b_k(t)}\geq L$ for all $k\in\Nat$.
Hence for any $k\in\Nat$, we have
\begin{align} 
\frac{1}{b_k^2(t)}\exp\biggl(\int_s^t2a_k(\tau)\,d\tau\biggr)\leq \frac{1}{L^2}\exp\bigl(2\la_0 T\bigr),
\end{align}
and
\begin{equation}
\norm{\uts}_{\op(H_s,H_t)}\leq \frac{1}{L} \exp(\la_0 T),\ \ 0\leq s < t\leq T.
\end{equation}
Therefore, by Theorem \ref{HS} there exists $M>0$ such that
\begin{align}
\norm{U(t,s)}_{\op(H_s;\cm_{t,s})}\leq \frac{M}{(t-s)^\frac{1}{2}}
\end{align}
Moreover, $H_{t_1}=H_{t_2}$ as Hilbert spaces for every $t_1,t_2\in[0,T]$.
Indeed, let $y\in H_{t_1}$, we know that
\begin{equation}
\Qm{t_1}e_k=\frac{b_k(t_2)}{b_k(t_1)}\Qm{t_2}e_k,\ \ \mbox{for any}\ k\in\Nat
\end{equation}
and
\begin{align}
\norm{y}_{H_{t_1}}\leq \frac{K}{L}\norm{y}_{H_{t_2}}.
\end{align}
So, Hypothesis \ref{3} is satisfied by $H_t$ for every $t\in[0,T]$ with $\theta=\frac{1}{2}$. Fixing $E=H_{t_0}$ with $t_0\in[0,T]$, Theorems \ref{schau1}(2) and \ref{schau2}(1) may be applied to obtain maximal H\"older regularity of the mild solution to \eqref{bwp}.

As an explicit example we can choose $a_k(t)=-k^2(t^k+1)$ and $b_k(t)=\sin(kt)+c$, $c>1$ for all $t\in[0,T]$. 

\subsection{Example 2}
Let $A(t)=a(t) I$, where $a$ is a continuous real valued map on $[0,T]$ and set $\norm{a}_\infty=a_0$. 
Hence $$U(t,s)=\exp\biggl(\int_s^ta(\tau)\,d\tau\biggr) I,\ \ (s,t)\in\Delta$$ is the strongly continuous evolution operator associated to the family $\{A(t)\}_{t\in[0,T]}$.

Let $\{B(t)\}_{t\in[0,T]}\subseteq\op(X)$ be a family of operators satisfying Hypothesis \ref{1} (2).
Since
\begin{equation}
Q(t,s)=\int_s^t \exp\biggl(2\int_\si^ta(\tau)\,d\tau\biggr)Q(\si)\,d\si\ \ (s,t)\in\Delta,\nonumber
\end{equation}
where $Q(\si)=B(\si)B(\si)^\star$, a sufficient and obvious condition for $\Tr{Q(t,s)}<+\infty$ is that $\Tr{Q(\si)}<+\infty$ for a.e. $\si\in[0,T]$ and $\si\longmapsto\Tr{Q(\si)}\in L^1(0,T)$. In this case, $\pst$ is a non autonomous generalization of  the classical Ornstein-Uhlenbeck semigroup widely used in the Malliavin Calculus. We refer to \cite{MR4011050} for Schauder theorems in such autonomous case.

Since $U(t,s)$ is a multiple of the identity, if $X$ is infinite dimensional $U(t,s)$ cannot map  $X$ into $\cm_{t,s}$ for any $(s,t)\in\Delta$ and $U(t,s)(H_s)=H_s$ for all $(s,t)\in\Delta$. We look for sufficient conditions such that the hypotheses of Theorem \ref{HS} are satisfied. 

In addition to the above assumptions on the trace of the operators $Q(\si)$, we require that for every $(s,t)\in\Delta$ there exists $C>0$ such that 
\begin{equation}\label{normbt}
\norm{B(s)}_{\op(X)}\leq C\norm{B(t)}_{\op(X)}.
\end{equation}
It follows
\begin{align}
\norm{\Qp{s}}_{\op (X)}^2=\norm{B(s)^\star}^2_{\op(X)}\leq C\norm{B(t)^\star}^2_{\op(X)}= C\norm{\Qp{t}}_{\op (X)}^2.
\end{align}
Then, by Proposition \ref{pseudo}, $H_s\subseteq H_t$ with continuous embedding and for $x\in H_s$ we have
\begin{align}
\norm{U(t,s)x}_{H_t}=\norm{\exp\biggl(\int_s^ta(\tau)\,d\tau\biggr)x}_{H_t}\leq e^{a_0T}\sqrt{C}\norm{x}_{H_s}.
\end{align}
Hence, we can take $E=H_s$ in Theorem \ref{regcm}.

We assume now that there exists $m>0$ such that 
\begin{equation}\label{normbt2}
0<m\leq\norm{B(t)}_{\op(X)}\leq K,\ \ \mbox{for any}\ t\in[0,T].
\end{equation}
In this case, for every $t_1,t_2\in[0,T]$
\begin{equation}
\norm{B(t_1)}_{\op(X)}\leq \frac{K}{m}\norm{B(t_2)}_{\op(X)}
\end{equation}
and \eqref{normbt} holds with $\displaystyle{C=\frac{K}{m}}$. Moreover $\ac_{t_1}=\ac_{t_2}$ as Hilbert spaces for every $t_1,t_2\in[0,T]$.
Hence Hypothesis \ref{3} is satisfied by choosing as $E$ any of such $\ac_t$ and we can apply Theorems \ref{schau1} and \ref{schau2}.

\subsection{Example 3}
We generalize now Example 2 of \cite{cerlun}. Let $T>0$ and let $\os\subseteq\R^d$ be a bounded open set with smooth boundary. We consider the evolution operator $U(t,s)$ in $X:=L^2(\os)$ associated to an evolution equation of parabolic type,

\begin{align}\label{parabolico}
\begin{cases}
u_t(t,x)=\oa(t)u(t,\cdot)(x),\ \ (t,x)\in(s,T)\times\os,\\
u(t,x)=0,\ \ (t,x)\in(s,T)\times\de\os,\\
\mathcal{B}(t)u(t,\cdot)(x)=0,\ \ (t,x)\in(s,T)\times\de\os.\\
\end{cases}
\end{align}
The differential operators $\oa(t)$ are defined by
\begin{equation}\label{operatoreat}
\oa(t)\ph(x)=\sum_{i,j=1}^d a_{ij}(t,x)D_{ij}\ph(x)+\sum_{i=1}^d a_{i}(t,x)D_{i}\ph(x)+a_0(t,x)\ph(x), t\in[0,T],\ x\in\os
\end{equation}
and the family of the boundary operators $\{\mathcal{B}(t)\}_{t\in[0,T]}$ is either of Dirichlet, Neumann or Robin type, namely
\begin{align}
\mathcal{B}(t)u=
\begin{cases}
\begin{aligned}
&u \ \ \quad \quad\quad \quad \quad\quad\quad \quad \quad \quad\quad \quad \quad \quad \mbox{(Dirichlet)},  \\[1ex]
&\displaystyle{\sum_{i,j=1}^da_{ij}(x,t)D_i u\,\nu_j}\quad \quad\quad\quad \quad \quad\, \ \ \mbox{(Neumann)},\\[1ex] 
&\displaystyle{\sum_{i,j=1}^da_{ij}(x,t)D_i u\,\nu_j}+b_0(x,t)u\quad \quad \mbox{(Robin)},
\end{aligned}
\end{cases}
\end{align}
 where $\nu=(\nu_1,...,\nu_d)$ is the unit outer normal vector at the boundary of $\Omega$. 
 
 We make the following assumptions.

\begin{ip}\label{A}
We assume $a_{ij}=a_{ji}$ and for some $\displaystyle{\rho\in\left(\frac{1}{2},1\right)}$, $a_{ij}, b_0\in C^{\rho,1}\left([0,T]\times\overline{\os}\right)$, $a_{i},a_0\in C^{\rho,0}\left([0,T]\times\overline{\os}\right)$. Moreover, we assume that there exists $\nu>0$ such that for all $\xi\in\R^d$
\begin{equation}
\sum_{i,j=1}^da_{ij}(t,x)\xi_i\xi_j\geq \nu\abs{\xi}^2,\ \ t\in[0,T],\ x\in\os.
\end{equation}
\end{ip}
Under Hypothesis \ref{A}, for all $t\in[0,T]$, we define
\begin{align} 
D(A(t))&=\left\{u\in H^2(\os):\ \mathcal{B}(t)u=0\right\}, \\ \nonumber
A(t)u&=\oa(t)u,\ \ u\in D(A(t)). \nonumber
\end{align}
 In \cite[Thm 2.4]{MR1780769} it is proved that for all $u_0\in X$ there exists a unique weak solution $u$ of \eqref{parabolico}.
 
 Setting $U(t,s)u_0=:u(t)$, $U(t,s)$ turns to be an evolution operator in $L^2(\os)$. Moreover in \cite{MR1780769} it is shown that $U(t,s)$ has an extension (still denoted by $U(t,s)$) to the whole space $L^1(\os)$ that belongs to $\op(L^1(\os);L^\infty(\os))$ and it is represented by 
\begin{equation}
U(t,s)\ph(x)=\int_\os k(x,y,t,s)\ph(y)\,dy,\ \ \ph\in L^1(\os)
\end{equation}
where for every $0\leq s<t\leq T$, $k(\cdot,\cdot,t,s)\in L^\infty(\os\times\os)$ and there exist $M,m>0$ such that 
\begin{equation}\label{kernel}
\abs{k(x,y,t,s)}\leq\frac{M}{(t-s)^{\frac{d}{2}}} e^{-\frac{\abs{x-y}^2}{m(t-s)}},\ \ x,y\in\os,\ 0\leq s<t\leq T.
\end{equation}
Moreover in \cite[Ex.\,2.8 and Ex.\,2.9]{Schnaubelt2004} it is proven that  the family $\{A(t)\}_{t\in[0,T]}$ satisfies the assumptions of \cite{MR945820, MR934508}. Then by \cite[Thm.\,2.3]{MR945820}, $\uts\in\op\left(D(A(s);D(A(t)))\right)\cap\op(X;D(A(t)))$ with norms bounded by a constant independent of $s$ and $t$.  Hence, there exists $C>0$, independent of $s$ and $t$, such that for every $s<t$ we have
\begin{align}
&\norm{\uts}_{\mathcal{L}(X)}\leq C,\label{SAT0T}\\
&\norm{\uts}_{\op(X;D(A(t)))}\leq \frac{C}{t-s},\label{AT1}\\
&\norm{\uts}_{\op(D(A(s));D(A(t)))}\leq C.\label{SAT2T}
\end{align} 
By \eqref{SAT0T} and \eqref{SAT2T} and interpolation arguments, $$\uts\in\op\left((X,D(A(s))_{\be,2};(X,D(A(t))_{\be,2})\right)$$ for all $0<\be<1$ and there exists a constant $C_1>0$ such that 
\begin{equation}
\norm{\uts}_{\op\left((X,D(A(s))_{\be,2};(X,D(A(t))_{\be,2})\right)}\leq C_1.
\end{equation}
We consider first Neumann and Robin boundary conditions for $A(t)$. Due to interpolation results and \cite[Thm.\,3.5, Thm.\,4.15]{MR1115176} for all $0\leq t\leq T$ and $0<\g<2$, we have
\begin{align}\label{guidetti0}
(X,D(A(t))_{\frac{\g}{2},2}&=\begin{cases}
 H^{\g}(\os) &\mbox{if}\ 0<\g<\frac{3}{2}\\[1ex]
 \left\{u\in H^{\frac{3}{2}}(\os)\ |\ \tilde{\mathcal{B}}(t)u\in\mathring{H}^{\frac{1}{2}}(\os)\right\} &\mbox{if}\ \g=\frac{3}{2} \\[1ex]
\left\{u\in H^{\g}(\os)\ |\ \mathcal{B}(t)u=0\right\}\quad &\mbox{if}\ \frac{3}{2}<\g<1
\end{cases},
\end{align}
where
\begin{align}
\tilde{\mathcal{B}}(t)u=
\begin{cases}
\begin{aligned}
&\displaystyle{\sum_{i,j=1}^da_{ij}(x,t)D_i u\,\tilde{\nu_j}}\quad \quad\quad\quad \quad \quad\, \ \ \mbox{(Neumann)}\\[1ex] 
&\displaystyle{\sum_{i,j=1}^da_{ij}(x,t)D_i u\,\tilde{\nu_j}}+b_0(x,t)u\quad \quad \mbox{(Robin)}
\end{aligned}
\end{cases},
\end{align} 
$\tilde{\nu}$ is a smooth enough extension of $\nu$ to $\overline{\os}$ and $\mathring{H}^{\frac{1}{2}}(\os)$ consists on all the elements $\ph\in H^{\frac{1}{2}}(\os)$ whose null extension outside $\overline{\os}$ belongs to $H^{\frac{1}{2}}(\R)$.

The simplest example of the family $\{B(t)\}_{t\in[0,T]}$ is given by $B(t)=(\id-\Delta)^{-\frac{\g}{2}}$ for all $t\in[0,T]$, where $\Delta$ is the realization of the Laplace operator in $L^2(\os)$ with Neumann boundary conditions. We have
\begin{align}\label{guidetti1}
D\left((\id-\Delta)^{\frac{\g}{2}}\right)&=(L^2(\os),D(\Delta))_{\frac{\g}{2},2}\nonumber \\
&=\begin{cases}
 H^{\g}(\os) &\mbox{if}\ 0<\g<\frac{3}{2}\\[1ex]
  \left\{u\in H^{\frac{3}{2}}(\os)\ |\ \iprod{\nabla u}{\tilde{\nu}}\in\mathring{H}^{\frac{1}{2}}(\os)\right\}\ &\mbox{if}\ \g=\frac{3}{2} \\[1ex]
\left\{u\in H^{\g}(\os)\ |\ \iprod{\nabla u}{\nu}_{|_{\de\os}}=0\right\}\quad &\mbox{if}\ \frac{3}{2}<\g<2
\end{cases},
\end{align}
where $\iprod{\nabla u}{\nu}_{|_{\de\os}}$ is the trace of $\iprod{\nabla u}{\nu}$ at $\de \os$.

It follows that for every $q\geq 2$, the embedding of $D\left((\id-\Delta)^{\frac{\g}{2}}\right)$ in $L^q(\os)$ is continuous for $\g\geq d \left(\frac{1}{2}-\frac{1}{q}\right)$. Hence for such choices of $\g$, $(\id-\Delta)^{-\frac{\g}{2}}\in\op(L^2(\os);L^q(\os))$. In the same way of \cite[Lemma\,4.3]{cerlun}, we can prove that if $q>d$ $Q(t,s)$ has finite trace for all $0\leq s<t\leq T$ and, in this case, Hypothesis \ref{1} holds.

We check now that the space $E=H^\g(\os)$ with $0<\g<\frac{3}{2}$ satisfies Hypothesis \ref{3}. In fact by \eqref{guidetti0} and \eqref{guidetti1} we have $E=\left(X,D(A(t))\right)_{\frac{\g}{2},2}=\ac_t$ for all $t\in[0,T]$. Then $\uts\in\op(E,\ac_t)$ for all $0\leq s<t\leq T$ with norm bounded by a constant independent of $s$ and $t$ and \eqref{stimaesp2} holds with $\al=0$. By Proposition \ref{cute}, Hypothesis \ref{3} is satisfied by $E=D\left(X,D(A(t))\right)_{\frac{\g}{2},2}$ with $\theta=\frac{1}{2}$. Notice that the spaces $\ac_t$ and $\left(X,D(A(t))\right)_{\frac{\g}{2},2}$ do not coincide in general for $\g\geq \frac{3}{2}$ due to the different boundary conditions.

 Summarizing, we need $q\geq2$, $q>d$ and $\displaystyle{d\left(\frac{1}{2}-\frac{1}{q}\right)\leq\g<\frac{3}{2}}$. Namely,
\begin{align}
\begin{cases}
q\geq 2,\ \ \ \ \ \ \ \frac{q-2}{2q}\leq\g<\frac{3}{2}\ \ \mbox{if}\ d=1\\[1ex]
q>2,\ \ \ \ \ \ \frac{q-2}{q}\leq\g<\frac{3}{2}\ \ \ \mbox{if}\ d=2\\[1ex]
q>3,\ \ \ \ \ \ 3\frac{q-2}{2q}\leq\g<\frac{3}{2}\ \ \mbox{if}\ d=3\\[1ex]
4<q<8,\ 2\frac{q-2}{q}\leq\g<\frac{3}{2}\ \ \mbox{if}\ d=4
\end{cases}.
\end{align}
So we can take
\begin{align} \label{sceltadgamma}
\begin{cases}
0\leq\g<\frac{3}{2}\ \ \ \ \mbox{if}\ d=1\\[1ex]
0<\g<\frac{3}{2}\ \ \ \ \mbox{if}\ d=2\\[1ex]
 \frac{1}{2}<\g<\frac{3}{2}\ \ \ \, \mbox{if}\ d=3\\[1ex]
1<\g<\frac{3}{2}\ \ \ \ \mbox{if}\ d=4
\end{cases}.
\end{align}

With the same choices of $\g$ and $d$, we can also take $E=L^2(\os)$. Indeed, since $U(t,s)\in \op\bigl(L^2(\os)\bigr)\cap\op\bigl(L^2(\os);D(A(t))\bigr)$, then by interpolation $\uts\in\op\left(L^2(\os),H^\g(\os)\right)$ if $\g<\frac{3}{2}$. Moreover there exists $C_2>0$ such that
\begin{equation}
\norm{\uts}_{\op\left(L^2(\os),H^{\g}(\os)\right)}\leq \frac{C_2}{(t-s)^{\frac{\g}{2}}}.
\end{equation}
Hence, choosing $E=L^2(\os)$ and noting that $\ac_t=H^\g(\os)$ for all $t\in[0,T]$, \eqref{stimaesp2} holds for $\displaystyle{\al=\frac{\g}{2}}$ and therefore by Proposition \ref{cute}, Hypothesis \ref{3}  holds with $\displaystyle{\theta=\frac{1}{2}+\frac{\g}{2}}$.
Moreover, since for all $p>2$ $L^p(\os)\subseteq L^2(\os)$ with continuous embedding, we can choose $E=L^p(\os)$. Hence for $0<\g<1$, \eqref{stimaesp2} holds for $\displaystyle{\al=\frac{\g}{2}}$ and therefore by Proposition \ref{cute}, Hypothesis \ref{3}  holds with $\displaystyle{\theta=\frac{1}{2}+\frac{\g}{2}}$.

We consider now Dirichlet boundary conditions for $A(t)$. In this case $D(A(t))=H^2(\os)\cap H^1_0(\os)$ for all $t\in[0,T]$. 

The simplest example of the family $\{B(t)\}_{t\in[0,T]}$ is given by $B(t)=(-\Delta)^{-\frac{\g}{2}}$ for all $t\in[0,T]$ where $\Delta$ is the realization of the Laplace operator in $L^2(\os)$ with Dirichlet boundary conditions. Due to interpolation results and \cite[Thm.\,3.5, Thm\,4.15]{MR1115176} we have
\begin{align}\label{guidetti2}
D\left((-\Delta)^{\frac{\g}{2}}\right)&=\left(L^2(\os),H^2(\os)\cap H^1_0(\os)\right)_{\frac{\g}{2},2}\nonumber \\
&=\begin{cases}
 H^{\g}(\os) &\mbox{if}\ 0<\g<\frac{1}{2}\\
 \mathring{H}^{\frac{1}{2}}(\os) &\mbox{if}\ \g=\frac{1}{2} \\
\left\{u\in H^{\g}(\os)\ |\ u_{|_{\de\os}}=0\right\}\quad &\mbox{if}\ \frac{1}{2}<\g<2
\end{cases}.
\end{align}

For every $q\geq 2$, the embedding of $D\left((-\Delta)^{\frac{\g}{2}}\right)$ in $L^q(\os)$ is continuous for $\g\geq d \left(\frac{1}{2}-\frac{1}{q}\right)$. Hence for such choices of $\g$, $(-\Delta)^{-\frac{\g}{2}}\in\op(L^2(\os);L^q(\os))$. In the same way of \cite[Lemma\,4.3]{cerlun}, we can prove that if $q>d$ $Q(t,s)$ has finite trace for all $0\leq s<t\leq T$ and, in this case, Hypothesis \ref{1} holds.

We check now that the space $E=D\left((-\Delta)^{\frac{\g}{2}}\right)$ with $0<\g<2$ satisfies Hypothesis \ref{3}. In fact by \eqref{guidetti0} and \eqref{guidetti2} we have $E=D\left((-\Delta)^{\frac{\g}{2}}\right)=\ac_t$ for all $t\in[0,T]$. Then $\uts\in\op(E,\ac_t)$ for all $0\leq s<t\leq T$ with norm bounded by a constant independent of $s$ and $t$ and \eqref{stimaesp2} holds with $\al=0$. By Proposition \ref{cute}, Hypothesis \ref{3} is satisfied by $E=D\left((-\Delta)^{\frac{\g}{2}}\right)$ with $\theta=\frac{1}{2}$. 

Summarizing, we need $q\geq2$, $q>d$ and $\displaystyle{d\left(\frac{1}{2}-\frac{1}{q}\right)\leq\g<2}$. Namely,
\begin{align}
\begin{cases}
q\geq 2,\ \ \ \ \ \ \ \frac{q-2}{2q}\leq\g<2\ \ \ \  \mbox{if}\ d=1\\[1ex]
q>2,\ \ \ \ \ \ \frac{q-2}{q}\leq\g<2\ \ \ \ \ \mbox{if}\ d=2\\[1ex]
q>3,\ \ \ \ \ \ 3\frac{q-2}{2q}\leq\g<2\ \ \ \ \mbox{if}\ d=3\\[1ex]
q>4,\  \ \ \ \ \ 2\frac{q-2}{q}\leq\g<2\ \ \ \ \mbox{if}\ d=4\\[1ex]
5<q<10,\ 5\frac{q-2}{2q}\leq\g<2\ \ \mbox{if}\ d=5\\[1ex]
\end{cases}.
\end{align}
So we can take
\begin{align} \label{sceltadgamma2}
\begin{cases}
0\leq\g<2\ \ \ \ \mbox{if}\ d=1\\[1ex]
0<\g<2\ \ \ \ \mbox{if}\ d=2\\[1ex]
 \frac{1}{2}<\g<2\ \ \ \, \mbox{if}\ d=3\\[1ex]
1<\g<2\ \ \ \ \mbox{if}\ d=4\\[1ex]
\frac{3}{2}<\g<2\ \ \ \ \mbox{if}\ d=5
\end{cases}.
\end{align}

We prove now that we can take $E=L^2(\os)$ for $0<\g<2$. Indeed, since $U(t,s)\in \op\bigl(L^2(\os)\bigr)\cap\op\bigl(L^2(\os);H^2(\os)\cap H^1_0(\os)\bigr)$, then by interpolation $\uts\in\op\left(L^2(\os),H^{\g}(\os)\right)$ for all $0<\g<2$. Moreover there exists $C_3>0$ such that
\begin{equation}
\norm{\uts}_{\op\left(L^2(\os),H^{\g}(\os)\right)}\leq \frac{C_3}{(t-s)^{\frac{\g}{2}}}.
\end{equation}
Hence, choosing $E=L^2(\os)$, \eqref{stimaesp2} holds for $\displaystyle{\al=\frac{\g}{2}}$ and therefore by Proposition \ref{cute}, Hypothesis \ref{3}  holds with $\displaystyle{\theta=\frac{1}{2}+\frac{\g}{2}}$.
Moreover, since for all $p>2$ $L^p(\os)\in L^2(\os)$ with continuous embedding, we can choose $E=L^p(\os)$. Hence for $0<\g<1$, \eqref{stimaesp2} holds for $\displaystyle{\al=\frac{\g}{2}}$ and therefore by Proposition \ref{cute}, Hypothesis \ref{3}  holds with $\displaystyle{\theta=\frac{1}{2}+\frac{\g}{2}}$.

Now, we fix $d=1$. In this case $\os$ is a bounded open interval $I$ and we can choose $\g=0$, namely $B(y)=Id$ for all $t\in[0,T]$. We take now $E=E_{\al,p}:=\left(L^2(I),H^2(I)\cap H^1_0(I)\right)_{\al,p}$ for every $0<\al<\frac{1}{2}$ and $p\in[1,+\infty)$. 

In this particular case we can characterize the Cameron Martin space $\cm_{ts}$ for every $0\leq s<t\leq T$ as the space $(L^2(I),H^2(I)\cap H^1_0(I))_{\frac{1}{2},2})=H^1_0(I)$ by \eqref{guidetti2}. The proof is in \cite{cerlun}.


Since $\uts\in\op\left(H^2(I)\cap H^1_0(I)\right)\cap\op\left(L^2(I)\right)$ with norms bounded by a constant independent of $s$, $t$ and \eqref{AT1} holds, we obtain that $\uts\in\op\left(H^1_0(I)\right)$ with norms bounded by a constant independent of $s$, $t$, $\uts\in\op\left(L^2(I);H^1_0(I)\right)$ and there exists $C_4>0$ such that 
\begin{equation}
\norm{\uts}_{\op\left(L^2(I);H^1_0(I)\right)}\leq\frac{C_4}{(t-s)^{\frac{1}{2}}}.
\end{equation}
 By reiteration, for every $0<\al<\frac{1}{2}$ and $p\in[1,+\infty)$, we get $$\left(L^2(I),H^1_0(I)\right)_{2\al,p}=\left(L^2(I),H^2(I)\cap H^1_0(I)\right)_{\al,p},$$
for every $0<\al<\frac{1}{2},\ p\in[1,+\infty)$.  

Hence, by interpolation $$U(t,s)\in\op\bigl(E_{\al,p};H^1_0(I)\bigr)$$ and there exists a constant $C_4>0$ independent of $s$ and $t$ such that
\begin{equation}
\norm{U(t,s)}_{\op(E_{\al,p};H^1_0(I))}\leq\frac{C_4}{(t-s)^{\frac{1}{2}-\al}}.
\end{equation}
Hypothesis \ref{3} is satisfied by $E_{\al,p}$ with $\theta=\frac{1}{2}-\al$.

We recall the following characterization of the spaces $E_{\al,p}$ as Besov spaces with (possibly) boundary conditions, 
\begin{equation}
E_{\al,p}=(L^q(I),W^{2,q}(I)\cap W^{1,q}_0(I))_{\al,p}=B^{2\al}_{q,p,0}(I),
\end{equation}
where
\begin{equation}
B^{2\al}_{q,p,0}(I)=
\begin{cases}
 B^{2\al}_{q,p}(I) &\mbox{if}\ \ 0<\al<\frac{1}{2q}\\[1ex]
 \mathring{B}^{\frac{1}{q}}_{p,q}(I) &\mbox{if}\ \al=\frac{1}{2q} \\[1ex]\left \{u\in B^{2\al}_{q,p}(I)\ |\ u_{|_{\de I}}=0\right\}\quad &\mbox{if}\ \frac{1}{2q}<\al<\frac{1}{2}
\end{cases},
\nonumber
\end{equation}
where $\mathring{B}^{\frac{1}{q}}_{p,q}(I)$ consists on all the elements $u\in B^{\frac{1}{q}}_{q,p}(I)$ whose null extension outside $\overline{I}$ belongs to $B^{\frac{1}{q}}_{p,q}(\R)$, and for $\al>\frac{1}{2q}$ $u_{|_{\de I}}$ means the trace of $u$ at $\de I$. See \cite[Thm.\,3.5, Thm\,4.15]{MR1115176}. We refer to \cite[Ch.\,2,4]{MR1328645} for the standard theory of the Besov spaces $B^{k}_{p,q}(\R),\ B^{k}_{p,q}(I)$.

\section*{Acknowledgement}

The author is a member of G.N.A.M.P.A. of the Italian Istituto Nazionale di Alta Matematica (INdAM) and has been partially supported by the G.N.A.M.P.A. project \textit{PROBLEMI  LINEARI  ELLITTICI  E  PARABOLICI  IN  DIMENSIONE INFINITA}.

\addcontentsline{toc}{section}{References}

\printbibliography

\end{document}